\documentclass[a4paper]{amsart}
\usepackage[utf8]{inputenc}

\title{Bernoulli disjointness}

\author{Eli Glasner}
\address{Department of Mathematics\\
     Tel Aviv University\\
     Tel Aviv\\
     Israel}
\email{glasner@math.tau.ac.il}

\author{Todor Tsankov}
\address{
  Institut de Math{\'e}matiques de Jussieu--PRG \\
  Universit\'e Paris Diderot, case 7012 \\
  8, place Aur{\'e}lie Nemours \\
  75205 Paris \textsc{cedex} 13 \\
  France
  -- and --
  D{\'e}partement de math{\'e}matiques et applications \\
  {\'E}cole normale sup{\'e}rieure \\
  75005 Paris \\
  France}
\curraddr{Institut Camille Jordan \\
  Universit\'e Claude Bernard Lyon 1 \\
  Universit\'e de Lyon \\
  43, boulevard du 11 novembre 1918 \\
  69622 Villeurbanne \textsc{cedex} \\
  France}
\email{tsankov@math.univ-lyon1.fr}

\author{Benjamin Weiss}
\address {Institute of Mathematics\\
 Hebrew University of Jerusalem\\
 Jerusalem\\
 Israel}
\email{weiss@math.huji.ac.il}

\author{Andy Zucker}
\address{
  Institut de Math{\'e}matiques de Jussieu--PRG \\
  Universit\'e Paris Diderot, case 7012 \\
  8, place Aur{\'e}lie Nemours \\
  75205 Paris \textsc{cedex} 13 \\
  France}
\email{andrew.zucker@imj-prg.fr}

\usepackage{hyperref}
\usepackage[initials,shortalphabetic]{amsrefs}

\usepackage[T1]{fontenc}
\usepackage[osf]{mathpazo}
\usepackage{amssymb}
\usepackage{eucal} 

\usepackage{my-macros}
\usepackage{enumitem}
\setlist[enumerate,1]{label=(\roman*), font=\normalfont}

\newcommand{\fin}[1]{\mathrm{Fin}(#1)}
\newcommand{\fib}{\mathrm{fib}}
\newcommand{\fA}{\mathfrak{A}}
\newcommand{\fB}{\mathfrak{B}}

\DeclareMathOperator{\Vis}{Vis}

\date{January 2019} 


\begin{document}

\begin{abstract}
 Generalizing a result of Furstenberg, we show that for every infinite discrete group $G$, the Bernoulli flow $2^G$ is disjoint from every minimal $G$-flow. From this, we deduce that the algebra generated by the minimal functions $\fA(G)$ is a proper subalgebra of $\ell^\infty(G)$ and that the enveloping semigroup of the universal minimal flow $M(G)$ is a proper quotient of the universal enveloping semigroup $\beta G$. When $G$ is countable, we also prove that for any metrizable, minimal $G$-flow, there exists a free, minimal flow disjoint from it and that there exist continuum many mutually disjoint minimal, free, metrizable $G$-flows. Finally, improving a result of Frisch, Tamuz, and Vahidi Ferdowsi and answering a question of theirs, we show that if $G$ is a countable icc group, then it admits a free, minimal, proximal flow.
\end{abstract}

\maketitle

\setcounter{tocdepth}{1} 
\tableofcontents

\section{Introduction}
\label{sec:introduction}

Let $G$ be an infinite discrete group, whose identity element we denote by either $e_G$ or just $e$. 
A \df{$G$-flow} (or a \df{$G$-dynamical system}) is a compact Hausdorff space $X$ equipped with an action of $G$ by homeomorphisms. We will often denote the fact that $G$ acts on $X$ by $G \actson X$. A \df{subflow} of $X$ is a closed, $G$-invariant subset. A $G$-flow $X$ is called \df{topologically transitive} if for any two non-empty open subsets $U, V \sub X$ there is a $g \in G$ with 
$gU \cap V \not=\emptyset$.
When $X$ is metrizable this is equivalent to the requirement that there is a \df{transitive point};
i.e. a point $x \in X$ whose orbit $Gx$ is dense.
By Baire's category theorem in a metric topologically transitive flow the set of transitive points is a dense $G_\delta$
subset of $X$.
The flow $X$ is called \df{minimal} if every orbit is dense or, equivalently, if $X$ has no proper subflows.

A central object in symbolic dynamics that one can construct for every discrete group $G$ is the \df{Bernoulli flow} $2^G := \{0,1\}^G$; the action of $G$ is given by $(g \cdot z)(h) = z(hg)$. Sometimes one also considers more general finite alphabets instead of $2 = \set{0, 1}$. The Bernoulli flow is often called the \df{Bernoulli shift} and its subflows are called \df{subshifts} or sometimes, \df{flows of finite type}.

The fundamental concept of disjointness of two dynamical systems, which is central to this paper, was introduced by Furstenberg in his seminal work \cite{Furstenberg1967}. Two $G$-flows $X$ and $Y$ are called \df{disjoint} (denoted by $X \perp Y$) if the only subflow of $X \times Y$ that projects onto $X$ and $Y$ is $X \times Y$ itself. For $X$ and $Y$ to be disjoint, at least one of them has to be minimal and if both of them are minimal, they are disjoint iff $X \times Y$ is minimal.

In \cite{Furstenberg1967}, Furstenberg showed, among many other beautiful results, that for the group of integers $\Z$, the Bernoulli flow $2^\Z$ is disjoint from every minimal flow and then applied this  to prove his famous Diophantine theorem: if $\Sigma$ is a non-lacunary semigroup of integers and $\alpha$ is an irrational, then $\Sigma\alpha$ is dense in the circle $\R/\Z$. The recent paper~\cite{Huang2018} characterizes the topologically transitive $\Z$-flows which are disjoint from all minimal flows.

Furstenberg also studied the smallest class of subshifts which contains all minimal subshifts and is closed under taking products, subflows and factors. Using the disjointness result mentioned above, he proved that this class does not contain the Bernoulli flow \cite{Furstenberg1967}*{Theorem~III.5} and conjectured \cite{Furstenberg1967}*{p.~41} that a similar result should hold if one starts with all minimal flows rather than the minimal subshifts. This was confirmed in \cite{Glasner1983} for $G = \Z$ and in this paper, we prove it for all discrete groups.

A more compact way to state that the Bernoulli flow $2^G$ is not a factor of a subflow of a product of minimal flows is to say that the closed algebra $\fA(G) \sub \ell^\infty(G)$ generated by the \df{minimal functions} is a proper subalgebra of $\ell^\infty(G)$. We explain this equivalence in Section~\ref{sec:min-functions}. 

We can state our main theorem as follows.
\vspace{1 mm}

\begin{theorem}
  \label{th:intro:disjoint-bernoulli}
  Let $G$ be an infinite discrete group. Then the following hold:
  \begin{enumerate}
  \item The Bernoulli flow $2^G$ is disjoint from every minimal $G$-flow.
    
  \item $\fA(G) \neq \ell^\infty(G)$.
  \end{enumerate}
\end{theorem}
\vspace{1 mm}

When $G$ is countable, Theorem~\ref{th:intro:disjoint-bernoulli}, together with some techniques from the theory of subshifts and Baire category methods, allows us to produce a multitude of disjoint minimal flows for every group $G$. While for specific groups, it is often not difficult to produce disjoint flows (for example, distal flows are always disjoint from proximal flows, and for the group of integers the collection of circle rotations provides a continuum of pairwise disjoint 
minimal flows), it had remained elusive to do this for all groups.

Recall that a flow $G \actson X$ is called \df{essentially free} if for every $g \in G$, the closed set $\set{x \in X : g \cdot x = x}$ has empty interior. When $G$ is countable and the flow is minimal, one can use the Baire category theorem to show that this is equivalent to the existence of a \df{free point} (that is, a point $x \in X$ such that the map $G \ni g \mapsto g \cdot x$ is injective), in fact, a dense $G_\delta$ set of free points. 
A $G$-flow is called \df{free} if every point is free. For example, the Bernoulli flow $2^G$ is essentially free but it is not free.

\vspace{1 mm}

\begin{theorem}
  \label{th:intro:free-disjoint}
  Let $G$ be a countable, infinite group. Then the following hold:
  \begin{enumerate}
  \item \label{i:thi:free-disjoint-1} For every minimal, metrizable $G$-flow $X$, there exists a minimal, metrizable, and free $G$-flow $Y$ such that $X \perp Y$.
    
  \item \label{i:thi:free-disjoint-2} There exist continuum many mutually disjoint, free, metrizable, minimal $G$-flows.
  \end{enumerate}
\end{theorem}
\vspace{1 mm}

Note that the condition that $X$ be metrizable is essential in item \ref{i:thi:free-disjoint-1} above. The universal minimal flow $M(G)$ (which is never metrizable for infinite, discrete $G$) is not disjoint from any minimal flow. In particular, this shows that if $T$ is a topological group for which $M(T)$ is metrizable (and there is an abundance of those, see, for example, \cite{Glasner2002a} or \cite{Pestov2006}), then there is no hope for a result analogous to Theorem~\ref{th:intro:free-disjoint} to hold. However, the analogue of Corollary~\ref{c:intro:env-semigroup} below is known to be true for some such groups (see~\cite{Bartosova2019a}).

We apply Theorem~\ref{th:intro:free-disjoint} to characterize the underlying space of the universal minimal flow $M(G)$ whenever $G$ is a countable, infinite group, generalizing results of~\cite{Balcar1990} and~\cite{Turek1991}. The definition of the \emph{Gleason cover} ``Gl'' appears at the end of Section~\ref{sec:bern-disj-count}.
\vspace{1 mm}

\begin{cor}
	\label{cor:intro:MGSpace}
	Let $G$ be a countable, infinite group. Then $M(G)\cong \mathrm{Gl}(2^\mathfrak{c})$.
\end{cor}
\vspace{1 mm}

Another way to say that some dynamic behaviors cannot be captured by the minimal flows uses the theory of the enveloping semigroup. For every topological group, there is a universal enveloping semigroup; if $G$ is discrete, this is just $\beta G$, the Stone--\v{C}ech compactification of $G$. Every topological group $T$ also admits a \df{universal minimal flow} $M(T)$, a minimal flow that maps onto every other minimal $T$-flow; for a discrete $G$, one can take, for example, any minimal subflow of $\beta G$. Then one can ask whether the canonical map from the universal enveloping semigroup to the enveloping semigroup of $M(G)$ (see Section~\ref{sec:min-functions} for the definition) is an isomorphism. 
Often attributed to Ellis and sometimes called the ``Ellis problem'', this question appears in Auslander (\cite{Auslander1988}, p.~120) and de Vries (\cite{Vries1993}, p.~391), and a negative answer was conjectured by Pestov (\cite{Pestov1998}, p.~4163) for all non-precompact topological groups. For discrete groups, this question is equivalent to the question whether $\fA(G)$ coincides with  $\ell^\infty(G)$ (see Section~\ref{sec:min-functions}), and thus in the second item of Theorem~\ref{th:intro:disjoint-bernoulli} we confirm the conjecture in this case.
\vspace{1 mm}

\begin{cor}
  \label{c:intro:env-semigroup}
  Let $G$ be an infinite discrete group and let $M(G)$ be its universal minimal flow. Then the canonical map from $\beta G$ to the enveloping semigroup of $M(G)$ is not an isomorphism.
\end{cor}
\vspace{1 mm}

Next we discuss briefly our strategy to prove Theorem~\ref{th:intro:disjoint-bernoulli}. First we reduce Theorem~\ref{th:intro:disjoint-bernoulli} to the \emph{existence} of an essentially free, minimal flow with certain combinatorial properties (that we call the separated covering property, SCP in short). Then for most of the paper, we assume that $G$ is countable and depending on its algebraic structure, we use different ideas to produce this flow. Recall that a group $G$ is \df{maximally almost periodic (maxap)} if it admits an injective homomorphism to a compact group; $G$ is \df{icc} if every element of $G$ other than $e_G$ has an infinite conjugacy class. We show that, up to taking a quotient by a finite normal subgroup, every group either admits an infinite maxap normal subgroup or is icc. In the first case, we use the equicontinuous action of the maxap subgroup to produce a free, minimal $G$-flow with
the SCP. In the second, we use a recent breakthrough of Frisch, Tamuz, and Vahidi Ferdowsi \cite{Frisch2019}, who constructed, for every icc group $G$, a faithful, metrizable, proximal flow. We improve \emph{faithful} to \emph{essentially free} and then show that every proximal flow has the SCP. 
We then indicate how the theorem for uncountable groups follows from the countable case.

In Section~\ref{sec:from-essent-freen}, we develop a general method to construct free minimal flows from essentially free ones, preserving many important properties (notably, disjointness). Using it, we prove the following theorem, answering a question asked in \cite{Frisch2019}.
\vspace{1 mm}

\begin{theorem}
  \label{th:intro:free-proximal}
  Every countable icc group admits a free, metrizable, proximal flow.
\end{theorem}
In fact, one can produce even continuum many mutually disjoint such flows (cf. Remark~\ref{rem:disjoint-proximal}).

In Section~\ref{sec:min-functions}, we study the algebra $\mathfrak{A}(G)$ generated by the minimal functions and prove Corollary~\ref{c:intro:env-semigroup}. In Section~\ref{sec:small-sets}, generalizing the results of \cite{Glasner1983} to arbitrary countable groups, we characterize the \df{interpolation sets} for the algebra $\fA(G)$. Finally, in Section~\ref{sec:ManyAuts}, we use our methods to produce for every countable group minimal flows with large groups of automorphisms.

Throughout the paper, $G$ denotes an infinite, discrete group. It is also assumed to be countable in Sections \ref{sec:sparse-cover-prox}, \ref{sec:some-prop-clos}, \ref{sec:bern-disj-count}, \ref{sec:small-sets}, and \ref{sec:ManyAuts}.

\subsubsection*{Addendum}
After learning about our results and after this paper had been circulated, Anton Bernshteyn~\cite{Bernshteyn2019p} found a different proof of the fact that the Bernoulli flow is disjoint from minimal flows using the Lovász Local Lemma.

Theorem~\ref{th:CantorAuts} has been improved by the fourth author in \cite{Zucker2019p}: there it is shown that for any two countable groups $G$ and $H$ with $G$ infinite, there exists a minimal $G$-flow on the Cantor space such that $H$ embeds in its automorphism group.

\subsection*{Acknowledgements}
\label{sec:acknowledgements}
We are grateful to Dana Barto\v{s}ov\'{a} 
whose talk  ``On a problem of Ellis and Pestov's conjecture'' held in Prague in the summer of 2016 rekindled our interest in this problem as well as Omer Tamuz for explaining some of the arguments in \cite{Frisch2019}. We would also like to thank the anonymous reviewers for a careful reading of the paper. Research on this project was partially supported by the NSF grant no.\ DMS 1803489, the ANR grants AGRUME (ANR-17-CE40-0026) and GAMME (ANR-14-CE25-0004), and Institut Universitaire de France. 
\vspace{2 mm}


\section{Bernoulli disjointness and the separated covering property}
\label{sec:bern-disj-sparse}

In this section, we reduce the problem of showing that all minimal flows of a group $G$ are disjoint from the Bernoulli flow to the existence of an essentially free $G$-flow with a certain combinatorial property.

If $D \sub G$ is finite, we will say that two sets $E_1, E_2 \sub G$ are \df{$D$-apart} if $D E_1\cap D E_2 = \emptyset$. We will say that a subset $E \sub G$ is \df{$D$-separated} if all of its elements are $D$-apart, i.e., $D g\cap D h = \emptyset$ for all $g, h \in E, \ g\neq h$.

\vspace{1 mm}

\begin{defn}
  \label{df:sparse-covering}
  Let $X$ be a minimal $G$-flow.
We will say that $X$ has the \df{separated covering property}, or the \emph{SCP},  
if  for every finite $D \sub G$ and every open, non-empty $U \sub X$, there exists a $D$-separated $S \sub G$ such that $S^{-1} U = X$.

We remark that by the compactness of $X$, $S$ can be taken to be finite.  
\end{defn}
\vspace{1 mm}

If $X$ is a $G$-flow, $x \in X$, and $U \sub X$, we will denote the visiting times of $x$ to $U$ by
\begin{equation*}
  \Vis(x, U) = \set{g \in G : g \cdot x \in U}.  
\end{equation*}

Recall that a point $x \in X$ is called \df{minimal}
if $\cl{G \cdot x}$ is a minimal flow. 
A subset $S \sub G$ is called \df{syndetic} if there is a finite set $F \sub G$ with $FS =G$. If $F\subseteq G$ is a finite set for which $FS = G$, we sometimes say that $S$ is \emph{$F$-syndetic}.
The first statement of the following lemma goes back to Gottschalk and Hedlund \cite{Gottschalk1955}.

\begin{lemma}\label{l:syndetic}
\begin{enumerate}
\item
A point $x \in X$ of a $G$-flow is minimal iff 
$\Vis(x, U)$ is syndetic for every open neighborhood $U$ of $x$.
\item
A subset $S \sub G$ is syndetic iff the subflow $\overline{G\cdot 1_S}$ of $2^G$ 
does not contain the constant function zero.
\item
A maximal $F$-separated set $L$ is $F^{-1}F$-syndetic.
\end{enumerate}
\end{lemma}

\begin{proof}
We prove only the last statement.
Let $g \in G$ be given. If $g \not \in F^{-1}FL$ then $Fg \cap F L = \emptyset$, whence
$L' = L\cup \{g\}$ is an $F$-separated set which properly contains $L$; a contradiction.
Thus $F^{-1}FL = G$. 
\end{proof}

In the next proposition, we will show that the SCP is exactly the dynamical property needed to prove that a minimal flow is disjoint from the Bernoulli flow. We need the following lemma.
\vspace{1 mm}

\begin{lemma}
	\label{l:BernoulliDensePoint}
	The Bernoulli flow $2^G$ has a free point whose orbit is dense.
\end{lemma}

\begin{proof}
Let $|G| = \kappa$, 
let $\{g_{\alpha +1} : \alpha < \kappa\}$ be an enumeration of $G \setminus \{e\}$,
and let $\mathcal{P}_f(G) = \{F_{\alpha +1}: \alpha < \kappa\}$ list the finite subsets of $G$. 
We will build $z\in 2^G$ free with dense orbit in $\kappa$-many stages, defining for each $\alpha \leq \kappa$ a function $z_\alpha \colon S_\alpha\to 2$ for some 
$S_\alpha\subseteq G$, with $S_\alpha\subseteq S_\beta$ for $\alpha < \beta$. 
When $\alpha < \kappa$, we will have $|S_\alpha| < \kappa$. 
Set $z_0 = \emptyset$, so also $S_0 = \emptyset$. 
At limit stages, set $z_\alpha = \bigcup_{\beta < \alpha} z_\beta$. 
Suppose $z_\alpha$ is defined for $\alpha < \kappa$. Letting $s_1, \ldots, s_n$ 
enumerate $2^{F_{\alpha +1}}$, find 
$h_{\alpha +1} \in G$ and
$g_{s_i}\in G$ with 
$\{F_{\alpha+1}  g_{s_1}, \ldots, F_{\alpha+1}  g_{s_n}, \{h_{\alpha +1}\}, \{h_{\alpha +1}g_{\alpha +1}\}, S_\alpha\}$ 
pairwise disjoint; this is possible because $|S_\alpha|< \kappa$. 

Set $S_{\alpha+1} = 
\bigcup\{F_{\alpha+1}  g_{s_1}, \ldots, F_{\alpha+1}  g_{s_n}, \{h_{\alpha +1}\}, \{h_{\alpha +1}g_{\alpha +1}\}, S_\alpha\}$. 
For $f\in F_{\alpha+1}$ and 
$i \leq n$, set $z_{\alpha+1}(fg_{s_i}) = s_i(f)$,
and $z_{\alpha+1}(h_{\alpha+1})=0,\  z_{\alpha+1}(h_{\alpha +1}g_{\alpha +1}) =1$. 
Continue until $z_\kappa$ is defined, then let $z \colon G\to \{0,1\}$ be any function extending $z_\kappa$. Since $z$ extends $z_{\alpha+1}$, we have that $g_{s_i}\cdot z|_{F_{\alpha+1}} = s_i$ and that $g_{\alpha+1}\cdot z(h_{\alpha+1})\neq z(h_{\alpha+1})$. It follows that $z$ is a free point with dense orbit as desired. 
\end{proof}
\vspace{1 mm}

\begin{remark}
When $G$ is countable and $X$ is metrizable, topologically transitive and essentially free, then 
both the sets $X_{\mathrm{tr}}$ and $X_{\mathrm{free}}$, of transitive and free points respectively,
are dense $G_\delta$ sets, hence so is $X_{\mathrm{tr}} \cap X_{\mathrm{free}}$.
In particular, this is true for the flow $X =2^G$.
\end{remark}
\vspace{1 mm}

For the next proposition, we introduce some notation. Given a finite $D\subseteq G$ and some $\alpha\in 2^D$, we let
 \begin{equation*}
N_\alpha = \set{z \in 2^G : z|_D = \alpha}
\end{equation*}
be the corresponding basic open neighborhood in $2^G$.
\vspace{1 mm}

\begin{prop}
  \label{p:SDOP-equiv}
  Let $X$ be a minimal $G$-flow. Then the following are equivalent:
  \begin{enumerate}
  \item 
For every finite $D\subseteq G$, there is a $D$-separated $S\subseteq G$ so that for every $x\in X$, $S\cdot x\subseteq X$ is dense.    
\item \label{i:SDOP-equiv:2} 
$X$ has the SCP.
  \item $X \perp 2^G$.
  \end{enumerate}
\end{prop}
\begin{proof}
  \begin{cycprf}
 \item[\impnext] This is clear.
    
  \item[\impnext] 
  
  To prove that $X \perp 2^G$, it suffices to show that for every $z \in 2^G$ with dense orbit and every $x \in X$, $\cl{G \cdot (x, z)} = X \times 2^G$. So fix $z_0\in 2^G$ with dense orbit and $x_0 \in X$. We also fix a finite $D \sub G$ and some $\alpha \in 2^D$. Letting $U \sub X$ be non-empty open, it suffices to find $g \in G$ such that $g \cdot z_0 \in N_\alpha$ and $g \cdot x_0 \in U$. By the hypothesis, there exists a finite $D$-separated set $S \sub G$ with $S^{-1}U = X$. As $S$ is $D$-separated, the sets $\set{Ds : s \in S}$ are pairwise disjoint, thus we can define $\beta \in 2^{DS}$ by $\beta(ds) = \alpha(d)$ for $d \in D, s \in S$. Let $h \in G$ be such that $h \cdot z_0 \in N_\beta$. In particular, this implies that $sh \cdot z_0 \in N_\alpha$ for all $s \in S$. Finally, choose $s \in S$ such that $sh \cdot x_0 \in U$. Now it is easy to check that $g = sh$ works.
  
\item[\impfirst] 
Fix $D\subseteq G$ finite and symmetric. Let $z_0 \in 2^G$ be a free point with a dense orbit. Let $V \ni z_0$ be open such that the sets $\set{dV : d \in D^2}$ are pairwise disjoint. This implies that $\Vis(z_0, V)$ is $D$-separated. For any $x \in X$, $\cl{G\cdot (x, z_0)}\subseteq X\times 2^G$ is a subflow with full projections; as $X \perp 2^G$, this implies that  $\cl{G\cdot (x, z_0)} = X\times 2^G$. It follows that $\Vis(z_0, V) \cdot (x, z_0) \subseteq X\times V$ is dense, so also $\Vis(z_0, V)\cdot x\subseteq X$ is dense.
  \end{cycprf}
\end{proof}
\vspace{1 mm}

\begin{cor}\label{cor:perp}
Let $X$ be a minimal $G$-flow and $Y$ a $G$-flow which has a free transitive point
(in particular, if $G$ is countable, one can take $Y$ which is minimal and essentially free).
If $X \perp Y$, then $X$ has the SCP.
\end{cor}

\begin{proof}
Apply the argument in the implication (iii) $\Rightarrow$ (i) of Proposition \ref{p:SDOP-equiv}.
\end{proof}
  
Notice that the SCP is inherited by factors: if $X$ and $Y$ are minimal, $\pi \colon X\to Y$ is a $G$-map, and $X$ has the SCP, then so does $Y$. The next proposition will show that for a fixed group $G$, if there is one essentially free minimal flow with the SCP, then any minimal $G$-flow has the SCP.
\vspace{1 mm}

When $f\colon X \to Y$ is a map and $A \subset X$, we sometimes write $f[A]$ for $\{f(x) : x \in A\}$. We are going to use the following standard fact: if $\pi \colon X \to Y$ is a $G$-map between minimal $G$-flows, then $\pi$ is \df{quasi-open}, i.e., $\pi[U]$ has non-empty interior for any non-empty, open $U \sub X$. To see this, let $V \sub \cl{V} \sub U$ with $V$ open, non-empty. Then by minimality, there is a finite $F \sub G$ such that $FV = X$, implying that $F \pi[\cl{V}] \supseteq F \pi[V] = Y$. This implies that the closed set $\pi[\cl{V}] \sub \pi[U]$ must have non-empty interior.

\begin{prop}
  \label{p:extensions}
  Let $\pi \colon X \to Y$ be a $G$-map between the minimal $G$-flows $X$ and $Y$ with $Y$ essentially free. Then if $Y$ has the SCP, so does $X$.
\end{prop}

\begin{proof}
  Let $D \sub G$ be given and let $V \sub X$ a non-empty open subset. As $\pi$ is quasi-open, $\Int{\pi[V]}$ is non-empty and we can replace $V$ by $\pi^{-1}(\Int{\pi[V]}) \cap V$. We then have (for the new $V$) that $\pi[V] = U \sub Y$ is non-empty and open. As $Y$ is essentially free, by shrinking $U$ and $V$ if necessary, we may further assume that the sets $\{d U : d \in D\}$ are pairwise disjoint.

  Pick $x_0 \in V$ and let $y_0 = \pi(x_0)$. By minimality, for each $x \in \pi^{-1}(y_0)$, there is an element $g_x \in G$ and an open $B_x \ni x$ such that $g_xB_x \sub V$. By compactness, there exist finitely many elements $x_1, \ldots, x_n \in \pi^{-1}(y_0)$, open sets $B_1, \ldots, B_n$, and elements $g_1, \ldots, g_n \in G$ such that for each $i$, we have that $x_i \in B_i$, $g_i B_i \sub V$, and $\bigcup_i B_i \supseteq \pi^{-1}(y_0)$.

  Note that the set $F = \{g_1, \ldots, g_n\}$ is $D$-separated.  Indeed, if $g_i\neq g_j$ and $d_i g_i = d_j g_j$ for some $d_i, d_j \in D$, then $d_i\neq d_j$ and $d_i g_i y_0 = d_j g_j y_0$. However, both $g_i y_0$ and $g_j y_0$ are in $U$ (as $g_ix_i, g_j x_j \in V$) and this contradicts the fact that $d_i U \cap d_j U = \emptyset$.

  Let $W \ni y_0$ be open and small enough so that $\pi^{-1}(W) \sub \bigcup_i B_i$; one can take $W = Y\setminus \pi[X\setminus \bigcup_i B_i]$. Use the SCP for $Y$ to find a $DF$-separated set $S$ such that $S^{-1}W = Y$. As $F^{-1}V \supseteq \bigcup_i B_i \supseteq \pi^{-1}(W)$, we now have that $S^{-1}F^{-1}V = X$. As the set $FS$ is $D$-separated, this completes the proof.
\end{proof}
\vspace{1 mm}

\begin{cor}
  \label{c:one-SDOP-enough}
If the group $G$ admits some minimal, essentially free flow with the SCP, then all minimal $G$-flows have the SCP.
\end{cor}

\begin{proof}
 Let $G \actson Y$ be a minimal essentially free $G$-flow with the SCP. 
 Let $X$ be any minimal $G$-flow. Let $Z \sub X \times Y$ be a minimal subflow.
 Then, as $Z$ extends $Y$, it has the SCP by Proposition \ref{p:extensions}.
 On the other hand $X$ is a factor of $Z$ and therefore it inherits the SCP.  
\end{proof}
\vspace{2 mm}


\section{Separated covering for maxap groups}
\label{sec:sparse-cover-maxap}

Recall that a $G$-flow $X$ is \emph{equicontinuous} if for every non-empty open $U\subseteq X\times X$ containing the diagonal, there is another open $V\subseteq X\times X$ containing the diagonal such that for any $(x, y)\in V$ and any $g\in G$, we have $(gx, gy)\in U$. 
A typical example of an equicontinuous action is when $G$ is a subgroup of a compact group $K$ and acts on $K$ by left multiplication. In the presence of equicontinuity, it is not difficult to show the separated covering property.
\vspace{1 mm}

\begin{lemma}
  \label{l:infinite-equic}
  Let $G \actson X$ be a minimal $G$-flow and suppose that there is an infinite subgroup $H \leq G$ such that the action $H \actson X$ is equicontinuous. Then $G \actson X$ has the SCP.
\end{lemma}

\begin{proof}
  Let $U \sub X$ be non-empty, open and let a finite $D \sub G$ be given. Since the action $H\actson X$ is equicontinuous, it is pointwise minimal, i.e.\ $\overline{H\cdot x}$ is a minimal $H$-flow for each $x\in X$. Fix $x_0 \in U$. Equicontinuity of the $H$-action allows us to find an open $V \ni x_0$, $V \sub U$ such that for all $h \in H$, if $h \cdot x_0 \in V$, then $hV \sub U$. Now recurrence implies that the set $R = \set{h \in H : h^{-1}V \sub U}$ is infinite. 
  Let $F \sub G$ be finite with $FV = X$. Write $F = \set{f_0, \ldots, f_{n-1}}$ and find $h_0, \ldots, h_{n-1} \in R$ such that the set $\set{f_0h_0, \ldots, f_{n-1}h_{n-1}}^{-1}$ is $D$-separated. Finally, $\bigcup_{i<n} f_ih_i U \supseteq FV = X$.
\end{proof}
\vspace{1 mm}

Recall that a group $G$ is called \df{maximally almost periodic (maxap)} if admits an injective homomorphism into a compact group (or, equivalently, a free equicontinuous flow). For example, residually finite groups and abelian groups are maxap. The main result of this section is the following.
\vspace{1 mm}

\begin{prop}

  \label{p:normal-maxap}
  Suppose that $G$ admits an infinite, normal subgroup $H$ which is maxap. Then $G$ admits a free, minimal flow with the SCP.
\end{prop}
\begin{proof}
  Let $H \actson K$ be a free, equicontinuous flow and let $G/H \actson Y$ be any free flow. Let $s \colon G/H \to G$ be a section for the quotient map $\pi \colon G \to G/H$ (i.e., such that $\pi \circ s = \id_{G/H}$) with $s(H) = e_G$ and define the cocycle $\rho \colon G \times G/H \to G$ 
 by $\rho(g, w) = s(gw)^{-1} g s(w)$. Define the \df{co-induced action} $G \actson K^{G/H}$ by:
  \begin{equation*}
    (g \cdot x)(w) = \rho(g^{-1}, w)^{-1} \cdot x(g^{-1}w).
  \end{equation*}
  It is easy to check that the restriction of this action to $H$ is equicontinuous as the action on each coordinate is the composition of the original action $H \actson K$ with an automorphism of $H$. Finally consider the product flow $G \actson Y \times K^{G/H}$ ($G$ acts on $Y$ via $\pi$) and let $Z$ be any minimal subflow. The flow $Z$ is free and satisfies the hypothesis of Lemma~\ref{l:infinite-equic}.
\end{proof}
\vspace{2 mm}


\section{Separated covering for proximal flows}
\label{sec:sparse-cover-prox}

Next we turn our attention to proximal flows, which are diametrically opposed to the equicontinuous flows considered in the previous section. We recall the definition below; the interested reader can consult \cite{Glasner1976} for more details.
\begin{defn}
  \label{defn:prox}
\begin{enumerate}
\item
Two points $x_1, x_2$ in a flow $G \actson X$ are called \df{proximal} if there are $y\in X$ and a net $g_i$ from $G$ with $g_ix_1\to y$ and $g_ix_2\to y$. Notice that if $x_1, x_2\in X$ are proximal and $X$ is minimal, then for any $y\in X$, we can find a net $g_i$ in $G$ as above. The flow $X$ is called \df{proximal} if all pairs of points in $X$ are proximal. 
When $X$ is proximal and minimal, then for any finite collection $x_1, \ldots, x_n\in X$ and any $y\in X$, there is a net $g_i$ from $G$ with $g_ix_k\to y$ for each $k \leq n$. 
An extension $\pi \colon X \to Y$ is \df{proximal} if all points $x_1, x_2$ with $\pi(x_1) = \pi(x_2)$ are proximal. A proximal extension of a proximal flow is always proximal.

\item
A $G$-flow $X$ is \df{strongly proximal} or a \df{boundary} 
if for every probability measure $\mu$ on $X$
there is a net $g_i \in G$ and a point $z \in X$ with
$\lim g_i \cdot \mu = \delta_z$. Clearly every strongly proximal flow is proximal.

\item
Every topological group $G$ admits a unique universal minimal proximal flow denoted $\Pi(G)$,
and  a unique universal minimal strongly proximal flow denoted $\Pi_s(G)$.
The latter is also called the \df{Furstenberg boundary} of $G$ and is sometimes denoted by $\partial_FG$.
\end{enumerate}
\end{defn}

For the remainder of this section, we assume that $G$ is countable. The groups for which $\Pi_s(G)$ is trivial are exactly the amenable groups (see \cite{Glasner1976}*{III.3.1. Theorem}). The groups for which $\Pi(G)$ is trivial are called \df{strongly amenable} and it was proved in \cite{Frisch2019} that those are exactly the groups with no icc quotients. Kalantar and Kennedy~\cite{Kalantar2017} showed that the action $G \actson \Pi_s(G)$ is free iff $G$ is \df{$C^*$-simple}, i.e., if the reduced $C^*$-algebra of $G$ is simple. Extending the results of \cite{Frisch2019}, we show below (cf. Section~\ref{sec:from-essent-freen}) that the action $G \actson \Pi(G)$ is free iff $G$ is icc.

Combining some known disjointness results and our techniques from Section~\ref{sec:bern-disj-sparse}, we can see that strongly proximal flows have the SCP. Recall first that minimal, strongly proximal flows are disjoint from minimal flows that admit an invariant measure (see \cite[Theorem III.6.1]{Glasner1976}). Next, by \cite{Weiss2012}, every group $G$ admits an essentially free minimal flow with an invariant measure. Finally, Corollary~\ref{cor:perp} implies that $\Pi_s(G)$ has the SCP. Thus, by Corollary~\ref{c:one-SDOP-enough}, if the action $G \actson \Pi_s(G)$ is free (i.e., if $G$ is $C^*$-simple), then all minimal actions of $G$ have the SCP. We do not provide more details because in Corollary~\ref{c:proximal-SDOP} we prove a more general result.

Recall that a point $x \in X$ is called \df{minimal} if $\cl{G \cdot x}$ is a minimal flow. 
We have the following basic fact.
\vspace{1 mm}

\begin{lemma}
  \label{l:alm-per-dense}
  Let $A$ be a finite alphabet. Then the minimal points in the Bernoulli flow $A^G$ are dense.
\end{lemma}

\begin{proof}
  Let $F \sub G$ be finite and let $\alpha \in A^F$ be a function. 
  We will show that the open set $N_\alpha = \set{z \in A^G : z|_F = \alpha}$ contains a minimal point. 
  Let $C \sub G$ be a maximal $F$-separated subset of $G$. 
  Let $a_0 \in A$. Define $z \in A^G$ by
  \begin{equation*}
    z(g) =
    \begin{cases}
      \alpha(f) & \text{ if } g = fc \text{ with } f \in F, c \in C, \\
      a_0        & \text{ if } g \notin FC.
    \end{cases}
  \end{equation*}
By Lemma \ref{l:syndetic}, $C$ is syndetic and moreover $\alpha$ appears syndetically in 
every element of $\cl{G \cdot z}$.  
Hence any minimal subset of $\cl{G \cdot z}$ intersects $N_\alpha$.
\end{proof}
\vspace{1 mm}

The following definition comes from \cite{Glasner1975} (there it is phrased in a different but equivalent manner).
\vspace{1 mm}

\begin{defn}
  \label{df:incontractible}
  A $G$-flow $X$ is called \df{incontractible} if for every $n \in \N$, the minimal points in $X^n$ are dense.
\end{defn}
\vspace{1 mm}

The proposition below essentially follows from \cite{Glasner1975}*{Theorem~4.2}. There one imposes the additional condition that $X$ is minimal, which is not necessary. Here we present an elementary proof.
\vspace{1 mm}

\begin{prop}
  \label{p:proximal-disj-incontr}
  Let $X$ be an incontractible $G$-flow and let $Y$ be a minimal, proximal $G$-flow. Then $X \perp Y$.
\end{prop}

\begin{proof}
  Let $W \sub X \times Y$ be a subflow with full projections. We will show that there exists $y \in Y$ such that $X \times \set{y} \sub W$. As $Y$ is minimal, this will show that $W = X \times Y$, as required.

  First we will show that for every open cover $U_1, \ldots, U_k$ of $X$, there exists a point $y_0 \in Y$ such that $(U_j \times \set{y_0}) \cap W \neq \emptyset$ for all $j$. Let $(x_1, \ldots, x_k) \in U_1 \times \cdots \times U_k$ be minimal  and let $y_j \in Y$ be such that $(x_j, y_j) \in W$. As $Y$ is proximal, there exists a net $(g_i)$ of elements of $G$ and $z \in Y$ such that $\lim_i g_iy_j = z$ for all $j \leq k$. By passing to a subnet, we may assume that $g_ix_j$ converges for all $j$ and we put $x_j' = \lim_i g_ix_j$. As $(x_1, \ldots, x_k)$ is minimal, there exists a net $(h_i)$ of elements of $G$ such that $\lim_i h_i x_j' = x_j$ for all $j$. By passing to a subnet, we may assume that $h_i z$ converges and we put $y_0 = \lim_i h_i z$. It is now clear that $(x_j, y_0) \in W$ for all $j$.

  Now for an open cover $\cO$, let $y_\cO$ be as above and let $y$ be a limit point of the net $(y_\cO : \cO \text{ open cover of } X)$ ordered by refinement. Then for every non-empty, open $U \sub X$ and every open neighborhood $V \ni y$, $W \cap (U \times V) \neq \emptyset$. Thus $X \times \set{y} \sub W$.
\end{proof}

Conversely, it can be shown that if  $X$ is a $G$-flow such that
$X \perp Y$ for every minimal proximal flow $Y$, and the set of 
minimal points is dense in $X$, then $X$ is incontractible.

\vspace{1 mm}

\begin{cor}
  \label{c:proximal-SDOP}
  The Bernoulli shift is disjoint from all minimal, proximal flows. In particular, all proximal flows have the SCP.
\end{cor}

\begin{proof}
  By Lemma~\ref{l:alm-per-dense}, minimal points are dense in $(A^G)^n = (A^n)^G$; thus $A^G$ is incontractible and we can apply Proposition~\ref{p:proximal-disj-incontr}.
\end{proof}
\vspace{2 mm}

A natural question one can ask is whether $M(G)$ can be proximal. We will see in Section~\ref{sec:ManyAuts} that for countable discrete groups this can never be the case.
\vspace{2 mm}


\section{Some basic facts about strongly irreducible subshifts}
\label{sec:some-prop-clos}

Our next goal is to show that a large class of countable groups admit essentially free proximal flows, so that we can apply the results of the previous section. Our results and techniques here are inspired by the recent paper of Frisch, Tamuz, and Vahidi Ferdowsi \cite{Frisch2019}, who showed that icc groups admit \emph{faithful}, minimal, proximal flows. In this, as well as in the next section, we assume that $G$ is countable.

Let $A$ be a finite alphabet with $|A|\geq 2$ and recall that $G \actson A^G$ by $(g \cdot z)(h) = z(hg)$. A \emph{subshift} is a subflow of $A^G$. The collection of subshifts forms a closed subspace of $K(A^G)$, the hyperspace of compact subsets of $A^G$ equipped with the Vietoris topology. To describe this topology, we introduce some notation. If $F\subseteq G$ is finite and $\alpha\in A^F$, set $N_\alpha = \{z\in A^G : z|_F = \alpha\}$. Given $K\in K(A^G)$ and a finite $F\subseteq G$, let
\begin{equation*}
  S_F(K) = \set{\alpha \in A^F : K \cap N_\alpha \neq \emptyset}.
\end{equation*}
Then a basis of neighborhoods at $K$ for the Vietoris topology is given by
\begin{equation*}
  \set[\big]{\set{K' \in K(A^G) : S_F(K') = S_F(K)} : F \sub G \text{ finite}}.
\end{equation*}

It will also be helpful to consider more general shifts whose alphabet is a Cantor space rather than a finite set. If $A$ is finite, we have an action $G\actson A^{\N\times G}$ via $(g\cdot z)(n, h) = z(n, hg)$. The definitions of subshift and Vietoris topology in this context are similar. In fact, we often view $A^G$ as a subshift of $A^{\N\times G}$ via the inclusion $i \colon A^G\to A^{\N\times G}$, where we set $i(z)(n, g) = z(g)$.

A subshift $Z \subseteq A^{\N\times G}$ is called \df{strongly irreducible} if for any $n \in \N$, there is a finite $D\subseteq G$ such that for any finite $E_1, E_2\subseteq G$ which are $D$-apart and any $z_1, z_2\in Z$, there is $x\in Z$ with 
$x|_{[n]\times E_i} = z_i|_{[n]\times E_i}$. 
(Here and below, we use the notation $[n] = \set{0, 1, \ldots, n-1}$.) We will say that $D$ \df{witnesses} the strong irreducibility of $Z$ for $n$. For a subshift $Z\subseteq A^G$, we say $Z$ is \df{strongly irreducible} iff $i[Z]$ is. In the case of a subshift $Z \sub A^G$, we will also say that $Z$ is \df{$D$-irreducible} instead of strongly irreducible with witness $D$. For example, both $A^{\N\times G}$ and $A^G$ are strongly irreducible. In many ways, strongly irreducible subshifts behave like the full Bernoulli shift but they offer more flexibility.

If $F \sub G$ is finite and $n \in \N$, an \df{$(n, F)$-pattern} is just a function $t \colon [n] \times F \to A$. If $z \in A^{\N \times G}$, we say that the pattern $t$ \df{appears in $z$} if there exists $g \in G$ such that $(g \cdot z)|_{[n] \times F} = t$. Similarly, if $V\subseteq G$ is finite, we say that $t\in A^{[n]\times F}$ appears in $u\in A^{[n]\times V}$ if there is $g\in G$ with $Fg\subseteq V$ and $u(k, fg) = t(k, f)$ for each $f\in F$ and $k\in [n]$.
\vspace{2 mm}

\begin{lemma}
	\label{l:strirred closure props}
	\begin{enumerate}
	\item 
	If $Y\subseteq A^{\N\times G}$ and $Z\subseteq B^{\N\times G}$ are strongly irreducible, then so is $Y\times Z\subseteq (A\times B)^{\N\times G}$. 
	\item 
	If $Y\subseteq A^{\N\times G}$ is strongly irreducible and $\phi \colon Y\to B^{\N\times G}$ is a $G$-map, then $Z := \phi[Y]$ is also strongly irreducible.
	\end{enumerate}
\end{lemma}

\begin{proof}
	The first item is clear. For the second, let $\phi \colon Y\to B^{\N\times G}$ be given, and fix $n \in \N$. We can find $N \in \N$ and finite symmetric $F\subseteq G$ so that for every $y\in Y$, $\phi(y)|_{[n]\times \{e_G\}}$ depends only on $y|_{[N]\times F}$. It follows that for every $g\in G$,  $\phi(y)|_{[n]\times \{g\}}$ depends only on $y|_{[N]\times Fg}$. If $D\subseteq G$ witnesses that $Y$ is strongly irreducible for $N$, then $DF$ will witness that $\phi[Y]$ is strongly irreducible for $n$. 
\end{proof}
\vspace{1 mm}

Sometimes, we want to choose explicit witnesses to irreducibility. Let $Z\subseteq A^{\N\times G}$ be strongly irreducible, $n \in \N$, and $D\subseteq G$ witness that $Z$ is strongly irreducible for $n$. Then if $E_1, E_2\subseteq G$ are finite and $D$-apart, $\alpha_i\in S_{(n,E_i)}(Z)$, and $F\subseteq G$ is finite with $E_i\subseteq F$, we let $\mathrm{Conf}_Z(F, \alpha_1, \alpha_2)$
be some fixed element of $S_{(n, F)}(Z)$ whose restriction to $[n]\times E_i$ is $\alpha_i$.


Following \cite{Frisch2019}, define $\mathcal{S}(A)\subseteq K(A^{G})$ to be the closure of the strongly irreducible subshifts with the finitely many constant configurations
removed. We define $\mathcal{S}(A^\N)\subseteq K(A^{\N\times G})$ analogously, removing the configurations which do not depend on the $G$ coordinate. Then $\mathcal{S}(A)\subseteq K(A^{G})$ is a compact space, and $\mathcal{S}(A^\N)\subseteq K(A^{\N\times G})$ is locally compact. These spaces are particularly well suited for Baire category arguments, a fact heavily exploited in \cite{Frisch2019}.
\vspace{1 mm}

\begin{lemma}
  \label{l:spaced-str-irred}
  For every finite $D$, there exists a strongly irreducible subshift $Y \sub 2^G$ such that for every $y \in Y$, the set $\set{g \in G : y(g) = 1}$ is $D$-separated and non-empty.
\end{lemma}

\begin{proof}
	By enlarging $D$ if necessary, we may assume $D$ to be symmetric. Let $Y$ be set of all $y \in 2^G$ such that $\{g\in G: y(g) = 1\}$ is a maximal $D$-separated set. Then $Y$ is a subshift such that for all $y \in Y$, the set $\set{g \in G : y(g) = 1}$ is $D$-separated and non-empty. We will show that $Y$ is $D^3$-irreducible. Let $E_1, E_2\subseteq G$ be finite and $D^3$-apart, and let $y_1, y_2\in Y$. First enlarge each $E_i$ to a set $E_i'\subseteq D^2E_i$ so that 
    \begin{enumerate}
  	\item 
  	If $g\in E_i'\setminus E_i$ we have $y_i(g) = 1$.
  	\item \label{i:l:spaced-str-irred:2}
  	For any $h\in E_i$ with $y_i(h) = 0$, there is $g\in D^2h\cap E_i'$ with $y_i(g) = 1$.
    \end{enumerate}
	One can do this by noting that any maximal $D$-separated set is $D^2$-syndetic by Lemma~\ref{l:syndetic}. Therefore for any $h$ as in item \ref{i:l:spaced-str-irred:2}, we can find $g\in D^2h$ with $y_i(g) = 1$ and add it to $E_i'$.
	Write $F_i = \{g\in E_i': y_i(g) = 1\}$. Notice that $F_1\cup F_2$ is a $D$-separated set and extend it to some maximal $D$-separated set $S$. Then $1_S\in Y$ and $(1_S)|_{E_i'} = y_i|_{E_i'}$. Indeed, it is clear that $F_i \sub E_i' \cap S$. For the reverse inclusion, note that if $h \in E_i' \sminus F_i$, then by the properties of $E_i'$, we have that $F_i \cap D^2h \neq \emptyset$; as $S$ is $D$-separated and $F_i \sub S$, this implies that $h \notin S$.
\end{proof}
\vspace{0 mm}

Given a subshift $Z\subseteq A^{\N\times G}$, a finite $F\subseteq G$, and $n \in \N$, we say that $Z$ is \emph{$(n, F)$-minimal} if for every $z\in Z$, every $t\in S_{(n, F)}(Z)$ appears in $z$. Notice that this occurs iff for some finite $V\subseteq G$, every $t\in S_{(n,F)}(Z)$ appears in every $u\in S_{(n,V)}(Z)$.

\begin{prop}
  \label{p:S-minimal-comeager}
  The set $\set{Z \in \cS(A^\N) : Z \text{ is minimal}}$ is dense $G_\delta$ in $\cS(A^\N)$. Similarly for $\cS(A)$. 
\end{prop}

\begin{proof}
  We only provide the proof for $\cS(A^\N)$, the other case being similar.
  
%
  
 Notice that a subshift $Z\subseteq A^{\N\times G}$ is minimal iff $Z$ is $(n,F)$-minimal for each $n \in \N$ and $F\subseteq G$ finite. By the remark before the statement of the proposition, being $(n,F)$-minimal is an open condition, so being minimal is $G_\delta$. To prove that it is dense in $\cS$, by the Baire category theorem, it suffices to show that for every strongly irreducible $Z'$, every $n \in \N$, and every finite $F \sub G$,  there exists a strongly irreducible $Z$ such that for every $z \in Z$, the set of $(n,F)$-patterns that appear in $z$ is exactly $S_{(n,F)}(Z')$.

  To do this, we employ the techniques of \cite{Frisch2019}. Fix $Z'$, $n$, and $F$. Let $V \supseteq F$ be finite, symmetric, and contain the identity. Suppose $V$ is large enough so that we can find $u \in Z'$ such that $u|_{[n]\times V}$ contains all $(n,F)$-patterns that appear in $Z'$. Suppose moreover that $V$ witnesses the strong irreducibility of $Z'$ for $n$. Let $Y \sub 2^G$ be given by Lemma~\ref{l:spaced-str-irred} with $D = V^5$.

  Next we define a continuous, equivariant map $\phi \colon Z' \times Y \to A^{\N\times G}$ as follows. Let $(z', y) \in Z' \times Y$ be given and denote $z = \phi(z', y)$. Let $g\in G$.
  
  \begin{itemize}
  \item If $m\geq n$, then define $z(m, g) = 0$ for each $g\in G$.
  \vspace{1 mm}
  
  \item If $g = kh$, where $y(h) = 1$ and $k\in V$, set $z(m,g) = u(m,k)$ for each $m< n$.
  \vspace{1 mm}
  
  \item If there is no $k \in V^3$ such that $y(kg) = 1$, define $z(m, g) = z'(m, g)$ for each $m< n$.
  \vspace{1 mm}
  
  \item If $g = kh$, where $y(h) = 1$ and $k \in V^3 \sminus V$, let
    \begin{equation*}
    z(m, g) = \mathrm{Conf}_{Z'}(V^5, (h\cdot z')|_{[n]\times (V^5\setminus V^3)}, u|_{[n]\times V})(m, k)  
    \end{equation*}
    for each $m< n$. The function $\mathrm{Conf}$ is defined after the proof of Lemma \ref{l:strirred closure props}.
  \end{itemize}

  Finally, let $Z = \phi(Z' \times Y)$. We will show now that for $z = \phi(z', y)$, the set of $(n, F)$-patterns that appear in $z$ is exactly $S_{(n, F)}(Z')$. 
  
  To see that every desired pattern does appear, find some $h\in G$ with $y(h) = 1$. Then $(h\cdot z)|_{[n]\times V} = u|_{[n]\times V}$, and every pattern from $S_{(n, F)}(Z')$ appears in $u|_{[n]\times V}$ by our choice of $u$. 
  
  To see that no new patterns appear, let $g\in G$. If $Fg\cap V^3h = \emptyset$ for every $h\in G$ with $y(h) = 1$, then $(g\cdot z)|_{[n]\times F} = (g\cdot z')|_{[n]\times F}\in S_{(n, F)}(Z')$. If there is $h\in G$ with $y(h) = 1$ and $Fg\cap V^3h\neq \emptyset$, then $Fg\subseteq V^5h$. We then have for each $f\in F$ and $m< n$ that
  \begin{equation*}
  (g\cdot z)(m, f) = \mathrm{Conf}_{Z'}(V^5, (h\cdot z')|_{[n]\times (V^5\setminus V^3)}, u|_{[n]\times V})(m, fgh^{-1}).
  \end{equation*} 
  So by the definition of $\mathrm{Conf}_{Z'}$, we have $(g\cdot z)|_{[n]\times F}\in S_{(n, F)}(Z')$ as desired.
\end{proof}
\vspace{1 mm}

Recall that a group is called \emph{icc} if the conjugacy class of every non-identity element is infinite. The following is the main result of \cite{Frisch2019}.
\vspace{1 mm}

\begin{theorem}[Frisch--Tamuz--Vahidi Ferdowsi]
  \label{th:FTVF}
  Let $G$ be an icc group and let $A$ be a finite alphabet. Then the set of minimal, proximal, faithful subshifts in $\cS(A)$ is dense $G_\delta$ in $\cS(A)$.
\end{theorem}
\vspace{1 mm}

The authors of \cite{Frisch2019} ask whether every icc group admits a free proximal flow. We answer this question in Section~\ref{sec:from-essent-freen} but as a preliminary step, here we show that one can easily get essential freeness.
\vspace{1 mm}

\begin{prop}
  \label{p:ess-free}
  If $G$ is icc or torsion free, then every non-constant strongly irreducible $Z\subseteq A^{G}$ is essentially free.
\end{prop}

\begin{proof}
  To see that a subshift $Z$ is essentially free, it suffices to show that for every $g \in G \sminus \set{e_G}$, finite $F \sub G$, and every $\alpha \in A^F$ for which $N_\alpha \cap Z \neq \emptyset$, there exists $z \in N_\alpha \cap Z$ and $h \in G$ such that $z(hg) \neq z(h)$. Suppose now that $Z$ is $D$-irreducible for some finite $D \sub G$.

  Suppose first that $G$ is icc. Then the conjugacy class of $g$ is infinite and therefore there exists $h \in G$ such that $h$, $hg$, and $F$ are $D$-apart. Let $z_1 \in Z \cap N_\alpha$, let $z_2, z_3 \in Z$ be such that $z_2(h) \neq z_3(hg)$ (those exist because $Z$ is non-constant). By $D$-irreducibility, there is $z \in Z$ such that $z|_F = z_1|_F$ (so that $z \in N_\alpha$), $z(h) = z_2(h)$, and $z(hg) = z_3(hg)$ (so that $z(h) \neq z(hg)$).

  Next suppose that $G$ is torsion-free. Then the order of $g$ is infinite and there exists $n \in \N$ and $h \in G$ such that $h$, $hg^n$, and $F$ are $D$-apart. Then, as above, there exists $z \in N_\alpha$ such that $z(h) \neq z(hg^n)$. Finally, there must exist $k < n$ such that $z(hg^k) \neq z(hg^{k+1})$, finishing the proof.
\end{proof}
\vspace{1 mm}

\begin{remark}
  \label{rem:ess-free}
  Note that it is not true that strongly irreducible subshifts are essentially free for all groups $G$. For example, if $G$ has a finite normal subgroup $N$, any strongly irreducible subshift of $G/N$ is also one for $G$ and it is not faithful. We do not know whether this is the only obstruction.
\end{remark}
\vspace{1 mm}

\begin{cor}
	\label{c:ess-free-minimal}
	Let $G$ be icc or torsion-free. Then $\set{Z \in \cS(A) : Z \text{ is essentially free}}$ is dense $G_\delta$ in $\cS(A)$. 
\end{cor}

\begin{proof}
  It follows from Proposition~\ref{p:ess-free} that the collection of essentially free subshifts is dense in $\cS$. That it is $G_\delta$ follows from the description in the first paragraph in the proof of this proposition.
\end{proof}
\vspace{1 mm}

We obtain the analogous statement for every group by considering shifts in $\cS(A^\N)$.
\vspace{1 mm}

\begin{prop}
	\label{p:ess-free-general}
	The set $\{Z\in \mathcal{S}(A^\N): Z \text{ is essentially free}\}$ is dense $G_\delta$ in $\mathcal{S}(A^\N)$. 
\end{prop}

\begin{proof}
	It remains to show that the set is dense. Suppose $Z'$ is strongly irreducible, let $n \in \N$, and let $F\subseteq G$ be finite. Define $\phi \colon Z'\times A^{\N\times G}\to A^{\N\times G}$ by 
	
	\begin{equation*}
      \phi(z', y)(k, g) =  
      \begin{cases}
        z'(k, g)  &\text{ if } k < n,\\
        y(k-n, g) &\text{ if } k\geq n.
      \end{cases}
	\end{equation*}
	
	Then $Z = \phi[Z'\times A^{\N\times G}]$ is strongly irreducible, essentially free, and we have $S_{(n, F)}(Z) = S_{(n, F)}(Z')$.
\end{proof}
\vspace{1 mm}

Combining Theorem~\ref{th:FTVF} with Proposition~\ref{p:ess-free}, we obtain the following.
\vspace{1 mm}

\begin{cor}
	\label{c:ExistsEssentiallyFree}
	Let $G$ be a countable icc group. 
	Then there exists a metrizable, essentially free, minimal, proximal $G$-flow $X$.
  \end{cor}

We will use Corollary~\ref{c:ExistsEssentiallyFree} to prove Theorem~\ref{th:intro:free-proximal} in Section~\ref{sec:from-essent-freen}.

\vspace{2 mm}


\section{Disjointness results}
\label{sec:bern-disj-count}

First we prove a slight strengthening of Proposition~\ref{p:SDOP-equiv}.
\vspace{1 mm}

\begin{lemma}
  \label{l:SDOP-implies-disj-str-irred}
  Let $X$ be a minimal $G$-flow with the SCP. Then $X$ is disjoint from any strongly irreducible subshift of $A^{\N \times G}$.
\end{lemma}

\begin{proof}
  Let $Y\subseteq A^{\N\times G}$ be a strongly irreducible subshift. Let $U\subseteq X$ be non-empty open. Let $n \in \N$, $F\subseteq G$ be finite, and $N_\alpha\cap Y\neq \emptyset$ for some $\alpha\in 2^{[n]\times F}$. Let $Z\subseteq X\times Y$ be a subflow with full projections. We need to show that $Z\cap (U\times N_\alpha) \neq \emptyset$ in order to conclude that $Z = X\times Y$.
    
  Suppose $D\subseteq G$ witnesses the strong irreducibility of $Y$ for $n$. Let $S\subseteq G$ be $DF$-separated such that $S^{-1}U = X$. Use the strong irreducibility of $Y$ to find $y\in Y$ with $g\cdot y|_{[n]\times F} = \alpha$ for every $g\in S$. As $Z\subseteq X\times Y$ has full projections, find $x\in X$ with $(x, y)\in Z$, then find $g\in S$ with $gx\in U$. It follows that $(gx, gy)\in Z\cap (U\times N_\alpha)$ as desired.
\end{proof}
\vspace{1 mm}

We can now prove our first main result.
\vspace{1 mm}

\begin{theorem}
  \label{th:main-BDJ}
  Let $G$ be a countable, infinite group. Then every minimal $G$-flow has the SCP and is thus disjoint from every strongly irreducible subshift. In particular every minimal $G$-flow is disjoint from the Bernoulli flow $2^G$.
\end{theorem}

We recall that every group $G$ admits a unique up to isomorphism universal minimal flow $M(G)$, which maps onto every minimal $G$-flow. One way to construct $M(G)$ is to consider the universal ambit $\beta G$
 (the Stone--\v{C}ech compactification of $G$); then every minimal subflow of $\beta G$ is isomorphic to $M(G)$. In particular, if $G$ is infinite, $M(G)$ is not metrizable.

We start with a lemma.
\vspace{1 mm}

\begin{lemma}
  \label{l:finite-index}
  Suppose that $N \lhd G$ is a finite, normal subgroup and $H = G/N$. Then if $M(H)$ has  the SCP, so does $M(G)$.
\end{lemma}

\begin{proof}
  Let $\theta \colon G \to H$ denote the quotient map and note that by the universal property of $\beta G$, it extends to a map $\beta G \to \beta H$, which in turn restricts to a map $\pi \colon M(G) \to M(H)$. It is not difficult to check that $\pi$ is continuous, open, and $|N|$-to-$1$, and for every open set $U \sub M(G)$, $\pi^{-1}(\pi[U]) = NU$. 
 
  Let $U \sub M(G)$ be open and let $x_0 \in U$. Let $D \sub G$ be finite and $N$-invariant. For every $a \in N$, choose $f_a \in G$ such that $f_a  \cdot a \cdot x_0 \in U$. Moreover, choose the $f_a$ in a way that the set $F = \set{f_a : a \in N}$ is $D$-separated. Let $V$ be a neighborhood of $x_0$ such that $f_a\cdot a\cdot V \sub U$ for all $a \in N$. Using the SCP for $M(H)$, let $S' \sub H$ be a $\theta(DF)$-separated subset of $H$ such that $S'^{-1}\pi[V] = M(H)$. Let $S \sub G$ be such that $\theta|_S \colon S \to S'$ is a bijection. Then $S$ is $DF$-separated and $S^{-1}NV = M(G)$. As $F$ is $D$-separated and $S$ is $DF$-separated, we obtain that $FS$ is $D$-separated. We finally note that
  \begin{equation*}
    (FS)^{-1}U = S^{-1}F^{-1}U \supseteq S^{-1}NV = M(G). \qedhere
  \end{equation*}
\end{proof}
\vspace{1 mm}

\begin{proof}[Proof of Theorem~\ref{th:main-BDJ}]
  We will show that $M(G)$ has the SCP. (This is sufficient by Lemma~\ref{l:SDOP-implies-disj-str-irred}.) Let $F$ be the \df{FC center} of $G$ consisting of all elements of $G$ with finite conjugacy classes. Note that $F$ is a characteristic subgroup of $G$. We will consider several cases.

  First, if $F$ is finite, then $G/F$ is icc and by Corollary~\ref{c:ExistsEssentiallyFree} and Corollary~\ref{c:proximal-SDOP}, $G/F$ admits an essentially free, minimal flow with the SCP. Thus by Proposition~\ref{p:extensions}, $M(G/F)$ has the SCP and by Lemma~\ref{l:finite-index}, so does $M(G)$.

  Suppose now that $F$ is infinite and let $Z$ be the center of $F$. Note that, as $Z$ is a characteristic subgroup of $F$, it is normal in $G$. We distinguish again two cases. If $Z$ is infinite, then we are done by Proposition~\ref{p:normal-maxap}. Finally, suppose that $Z$ is finite and let $F' = F/Z$ and $G' = G/Z$. We will check that $F'$ is residually finite. Let $\set{C_i : i \in \N}$ enumerate the non-identity conjugacy classes of $F$ and note that each $C_i$ is finite. Define $\phi \colon F \to \prod_i \Sym(C_i)$ by $\phi(g) \cdot (x_i)_{i \in I} = (g x_i g^{-1})_{i \in I}$ and note that $\ker \phi = Z$. Thus $\phi$ factors to an embedding of $F'$ into the profinite group $\prod_i \Sym(C_i)$. As $F'$ is infinite and normal in $G'$, Proposition~\ref{p:normal-maxap} implies that $M(G')$ has the SCP and thus, by Lemma~\ref{l:finite-index}, so does $M(G)$.
\end{proof}
\vspace{1 mm}

Next, we show how to produce minimal flows disjoint from any given \emph{metrizable} minimal flow. We will use the following method to show that a given set is $G_\delta$. If $X$ and $Y$ are compact Hausdorff spaces and $Z \sub X \times Y$ is a $G_\delta$ set, then the set $B \sub X$ defined by
\begin{equation*}
x \in B \iff \forall y \in Y \ (x, y) \in Z
\end{equation*}
is also $G_\delta$. To see this, note that $(X\times Y)\setminus Z$ is $K_\sigma$, so letting $\pi_X \colon X\times Y\to X$ be the projection, the set $\pi_X[(X\times Y)\setminus Z]$ is $K_\sigma$, and 
$B = X\setminus \pi_X[(X\times Y)\setminus Z]$.

Recall that if $X$ is a compact space, we denote by $K(X)$ the hyperspace of closed, non-empty subsets of $X$ equipped with the Vietoris topology. For basic facts about this topology, we refer the reader to \cite{Kechris1995}*{4.F}.
\vspace{1 mm}


\begin{prop}
  \label{p:disjointness}
  Let $G$ be a countable, infinite group and let $X$ be a metrizable minimal $G$-flow. Then the set of subshifts disjoint from $X$ is $G_\delta$ in $K(A^{\N\times G})$.
\end{prop}

\begin{proof}
  Let $Z \sub A^{\N\times G}$ be a subshift. By definition, $Z$ is disjoint from $X$ iff
  \begin{equation*}
    \forall K \in K(X \times A^{\N\times G}) \quad K \text{ is not $G$-invariant} \Or \pi_2(K) \nsupseteq Z \Or K \supseteq X \times Z.
  \end{equation*}
  Here $\pi_2$ denotes the projection to the second coordinate. The first two conditions after the quantifier are open and the third is closed. As $X$ is metrizable, so is $K(X \times A^{\N\times G})$, and closed sets in the latter space are $G_\delta$. Thus the whole condition after the quantifier is $G_\delta$ and we can conclude by applying the remark preceding the proposition.
\end{proof}
\vspace{1 mm}

\begin{cor}
  \label{c:disjointness}
  For any countable, infinite group $G$ and any metrizable minimal $G$-flow $X$, the set
  \begin{equation*}
    \set{Z \in \cS(A^\N) : Z \text{ is minimal and essentially free and } Z \perp X}
  \end{equation*}
  is dense $G_\delta$ in $\cS(A^\N)$.
\end{cor}

\begin{proof}
  By Proposition~\ref{p:S-minimal-comeager}, the minimal subshifts form a dense $G_\delta$ subset of $\cS(A^\N)$; by Proposition~\ref{p:ess-free-general}, so do the essentially free ones. By Proposition~\ref{p:disjointness}, being disjoint from $X$ is a $G_\delta$ condition and by Theorem~\ref{th:main-BDJ}, it is dense in $\cS(A^\N)$. By the Baire category theorem, the intersection of those three sets is dense $G_\delta$.
\end{proof}
\vspace{1 mm}

We also have the following corollary, generalizing a result of \cite{Frisch2017} (where it was proved for amenable groups).
\vspace{1 mm}

\begin{cor}
	\label{c:str-irr-not-Gdelta}
	Let $G$ be a countable, infinite group. Then the set of strongly irreducible subshifts of $G$ is not $G_\delta$.
  \end{cor}
  \begin{proof}
    Suppose that the set of strongly irreducible subshifts is $G_\delta$. Then it is comeager in $\cS(A)$ and by Proposition~\ref{p:S-minimal-comeager} and the Baire category theorem, it follows that there exists a subshift which is both minimal and strongly irreducible. This contradicts Theorem~\ref{th:main-BDJ}.
  \end{proof}
  
\vspace{2 mm}


A collection $\set{X_i : i \in I}$ of minimal $G$-flows is called \df{mutually disjoint} if the product $\prod_{i \in I} X_i$ is minimal. Note that a collection is mutually disjoint iff every finite subcollection is. For the next corollary, we will need the fact that the spaces $\cS(A)$ and $\cS(A^\N)$ are perfect.
\vspace{1 mm}

\begin{lemma}
	\label{l:SA-perfect}
	The spaces $\cS(A)$ and $\cS(A^\N)$ do not have isolated points.
\end{lemma}

\begin{proof}
	We give the argument only for $\cS(A^\N)$, the other case being similar. Given a strongly irreducible $Z\subseteq A^{\N\times G}$, we know by Lemma~\ref{l:SDOP-implies-disj-str-irred} and Theorem \ref{th:main-BDJ} that $Z$ cannot be minimal. Therefore there are $n\in \N$ and $F\subseteq G$ finite such that for some $z\in Z$, the set of $(n, F)$-patterns appearing in $z$ is a strict subset of $S_{(n,F)}(Z)$. However, the proof of Proposition \ref{p:S-minimal-comeager} shows that there is some strongly irreducible $Z'$ with $S_{(n,F)}(Z') = S_{(n,F)}(Z)$ and so that every $(n,F)$-pattern in $S_{(n,F)}(Z')$ appears in every $z'\in Z$. In particular, $Z\neq Z'$, so $\cS(A^\N)$ has no isolated points. 
\end{proof}
\vspace{1 mm}

\begin{cor}
  \label{c:continuum-many-disjoint}
  For every countable, infinite group $G$, there exist continuum many mutually disjoint, essentially free, minimal, metrizable $G$-flows.
\end{cor}
\begin{proof}
  Let $\Xi$ denote the Polish space of essentially free, minimal subshifts in $\cS(A^\N)$ and let
  \begin{equation*}
    R_n = \set{(Z_1, \ldots, Z_n) \in \Xi^n : Z_1, \ldots, Z_n \text{ are mutually disjoint}}.
  \end{equation*}
  By Corollary~\ref{c:disjointness}, we have that
  \begin{equation*}
    \forall (Z_1, Z_2, \ldots, Z_{n-1}) \in R_{n-1} \ \forall^* Z \in \Xi \quad Z \perp Z_1 \times \cdots \times Z_{n-1}
  \end{equation*}
  (here $\forall^*$ denotes the category quantifier ``for comeagerly many''). Now the Kuratowski--Ulam theorem (see, for example, \cite{Kechris1995}*{8.41}) and an easy induction imply that $R_n$ is comeager in $\Xi^n$ for every $n$. Finally, Lemma~\ref{l:SA-perfect} and Mycielski's theorem (see, for example, \cite{Kechris1995}*{19.1}) give us a Cantor set $C \sub \Xi$ of mutually disjoint subshifts.
\end{proof}

\begin{remark}
  \label{rem:disjoint-proximal}
  The corollary above is quite flexible and by varying $\Xi$, we can add more comeager properties if desired. For example, for icc groups, this allows to construct continuum many disjoint, essentially free (or, using the method of Section~\ref{sec:from-essent-freen}, even free), minimal, \emph{proximal} flows. To see that proximality is a dense condition in $\cS(A^\N)$, note that $\bigcup_n \cS(A^n)$ is dense in $\cS(A^\N)$, so we can apply Theorem~\ref{th:FTVF}. (To view $\cS(A^n)$ as a subset of $\cS(A^\N)$, just put some fixed letter $a_0$ in the unused slots.) Proximality is also $G_\delta$, as a subshift $X \in \cS(A^\N)$ is proximal iff
  \begin{equation*}
    \forall x, y \in A^{\N \times G} \ \forall \eps > 0 \quad x, y \in X \implies \exists g \in G\ d(g \cdot x, g \cdot y) < \eps,
  \end{equation*}
  where $d$ is some fixed metric on $A^{\N \times G}$.
\end{remark}

\vspace{1 mm}

We can use the previous corollary to solve an open problem regarding the homeomorphism type of the space $M(G)$ given a countable group $G$. Recall that if $X$ is a topological space, the \emph{$\pi$-weight} of $X$, denoted $\pi w(X)$, is the least cardinal $\kappa$ for which there exists a family $\{U_i: i< \kappa\}$ of open subsets of $X$ such that for every open $V$, there is $i< \kappa$ with $U_i\subseteq V$. Balcar and B\l aszczyk show in~\cite{Balcar1990} that the space $M(G)$ is homeomorphic to the \emph{Gleason cover} of the space $2^{\pi w(M(G))}$. Here the \emph{Gleason cover} of a compact space $X$, denoted $\mathrm{Gl}(X)$, is just the Stone space of the regular open algebra of $X$ (we will revisit this construction in Section~\ref{sec:from-essent-freen}). For a countable infinite group $G$, the largest possible value of $\pi w(M(G))$ is $\mathfrak{c}$, the cardinality of the continuum; to show that this maximum is attained, it suffices to find some minimal flow $X$ with $\pi w(X) = \mathfrak{c}$. Balcar and B\l aszczyk provide such a flow for $G = \Z$, and Turek~\cite{Turek1991} gives a construction for any countable abelian $G$. 

Now let $G$ be any countable infinite group. Corollary~\ref{c:continuum-many-disjoint} provides a family $\{X_i: i< \mathfrak{c}\}$ of essentially free, minimal, metrizable flows with $X = \prod_i X_i$ minimal. In particular, each $X_i$ has a nontrivial underlying space. It follows that $\pi w(X) = \mathfrak{c}$. This proves Corollary~\ref{cor:intro:MGSpace}, which we restate below. 
\vspace{1 mm}

\begin{cor}
	\label{c:MGSpace}
	Let $G$ be a countable, infinite group. Then $M(G) \cong \mathrm{Gl}(2^\mathfrak{c})$.
\end{cor}


\section{Uncountable groups}
\label{sec:uncountable-groups}

In this section, we indicate how to remove the countability assumption on $G$ in Theorem~\ref{th:main-BDJ}.
\vspace{1 mm} 

\begin{theorem}
	\label{th:UncountableGroups}
	Let $G$ be a discrete group. Then every minimal flow has the SCP and is disjoint from the Bernoulli flow $2^G$.
\end{theorem}

\begin{proof}
	We use a technique introduced by Ellis in~\cite{Ellis1978}. Let $X$ be a minimal, free $G$-flow. If $K\subseteq G$ is a countable subgroup and $\rho$ is a continuous pseudometric on $X$, then 
	$$R(K, \rho) := \{(x, y)\in X^2: \rho(kx, ky) = 0 \text{ for all } k\in K\}$$
	is a closed equivalence relation, and the quotient $X/R(K, \rho)$ is a metrizable $K$-flow. It was proved in \cite{Ellis1978}*{Proposition~1.6} that for any fixed $\rho$ and any countable subgroup $K\subseteq G$, there is a countable subgroup $L$ with $K\subseteq L\subseteq G$ such that $X/R(L, \rho)$ is minimal. Also, it follows from the proof of \cite{MatteBon2020}*{Proposition~2.9} (see also Proposition~\ref{p:EssFreeToFree} below) that for any fixed $K$ and any $\rho_1$, there is a continuous pseudometric $\rho_2$ such that $R(K, \rho_2) \sub R(K, \rho_1)$ with $X/R(K, \rho_2)$ a free $K$-flow. Alternating these steps countably many times, we see that for every countable subgroup $K\subseteq G$ and every continuous pseudometric $\rho_1$ on $X$, there are a countable supergroup $L\supseteq K$ contained in $G$ and a finer continuous pseudometric $\rho_2$ with $R(\rho_2, L) \sub R(\rho_1, K)$ such that $X/R(L, \rho_2)$ is minimal and free. 
	
	Whenever $K\subseteq L$ are countable subgroups of $G$ and $\rho_1$ and $\rho_2$ are continuous pseudometrics on $X$ with $R(\rho_2, L) \sub R(\rho_1, K)$, there is a natural $K$-map $\pi(L, \rho_2, K, \rho_1) \colon X/R(L, \rho_2)\to X/R(K,\rho_1)$. We can now write 
	$$X = \varprojlim X/R(L, \rho)$$
	where the inverse limit is taken over all pairs $(L, \rho)$ with $X/R(L, \rho)$ a minimal free $L$-flow. For each pair $(L, \rho)$ appearing in the inverse limit, we also have a natural $L$-map $\pi(L, \rho) \colon X\to X/R(L, \rho)$.
	
	Now let $D\subseteq G$ be finite and $U\subseteq X$ be non-empty open. We may assume that $U = \{x\in X: \pi(L, \rho)(x)\in V\}$ for some non-empty open $V\subseteq X/R(L, \rho)$ with $D\subseteq L$. Use Theorem~\ref{th:main-BDJ} to find a finite $D$-separated $S\subseteq L$ with $S^{-1}V = X/R(L, \rho)$. It follows that $S^{-1}U = X$ as desired. 
  \end{proof}

  \begin{remark}
    \label{rem:elem-substructure}
    The reduction to the countable case can also be carried out by an elementary substructure argument using the Löwenheim--Skolem theorem and the fact that minimality, freeness, and the SCP can be expressed by first-order sentences in appropriate structures.
  \end{remark}
  
\vspace{1 mm}


\section{From essential freeness to freeness}
\label{sec:from-essent-freen}

In this section, we discuss a general method to produce free flows out of essentially free ones, keeping many important properties. This will allow us to improve several of the results of the previous sections.

If $Y$ and $X$ are $G$-flows, $\phi \colon Y\to X$ is a surjective $G$-map, and $B \sub Y$, we define the \df{fiber image of B} by
\begin{equation*}
  \phi_{\fib}(B):= \{x\in X: \phi^{-1}(\{x\})\subseteq B\}.
\end{equation*}
Note that if $B$ is open, $\phi_\fib(B)$ is also open. We say that $\phi$ is \emph{highly proximal} if $X$ is minimal and $\phi_{\fib}(B)$ is non-empty for every non-empty open $B\subseteq Y$. Equivalently, $\phi$ is highly proximal iff for every $x \in X$ and $U \sub Y$ open, there exists $g \in G$ such that $g \cdot \phi^{-1}(\set{x}) \sub U$.

The following are some basic properties of highly proximal extensions.
\vspace{1 mm}

\begin{prop}
  \label{p:highly-proximal-basic}
  The following statements hold:
  \begin{enumerate}
  \item \label{i:p:hpb:1} Let $\phi \colon Y \to X$ be a highly proximal extension. Then $Y$ is minimal and $\phi$ is a proximal extension. In particular, if $X$ is proximal, then $Y$ is also proximal.

  \item \label{i:p:hpb:2} Let $\phi \colon Y \to X$ be a highly proximal extension with $X$ strongly proximal. Then $Y$ is strongly proximal.
    
  \item \label{i:p:hpb:3} Let $X_1, X_2, \ldots, X_n$ be minimal $G$-flows which are mutually disjoint and let $X_i' \to X_i$ be highly proximal extensions. Then $X_1', \ldots, X_n'$ are also mutually disjoint.
  \end{enumerate}
\end{prop}

\begin{proof}
  \ref{i:p:hpb:1} Let $y_1, y_2 \in Y$ with $\phi(y_1) = \phi(y_2) = x$. Let $B\subseteq Y$ be non-empty open. Then $\phi_{\fib}(B)$ is also non-empty open, so by minimality, we can find $g\in G$ with $gx\in \phi_{\fib}(B)$. In particular, we have $gy_1, gy_2\in B$, showing that $y_1$ and $y_2$ are proximal and that $Y$ is minimal.

  \ref{i:p:hpb:2} It is enough to show that for every open $U \sub Y$, every probability measure $\mu$ on $Y$, and every $\eps > 0$, there exists $g \in G$ such that $(g_*\mu)(U) > 1 - \eps$. Let $\nu = \phi_*\mu$ and $V = \phi_\fib(U)$. As $X$ is strongly proximal and minimal and $V \neq \emptyset$, there exists $g \in G$ such that $(g_* \nu)(V) > 1 - \eps$. As $\phi^{-1}(V) \sub U$, we are done.

  \ref{i:p:hpb:3} It is easy to check that the extension $\prod_i X_i' \to \prod_i X_i$ is highly proximal. By hypothesis, $\prod_i X_i$ is minimal; \ref{i:p:hpb:1} implies that so is $\prod_i X_i'$.
\end{proof}
\vspace{1 mm}

Auslander and Glasner~\cite{Auslander1977} prove the existence and uniqueness of a \emph{universal highly proximal extension} for minimal flows. Given a minimal flow $X$, there is a $G$-flow $\tS_G(X)$ and a highly proximal $G$-map $\pi_X \colon \tS_G(X)\to X$ through which every other highly proximal map to $X$ factors. An explicit construction of $\tS_G(X)$ is provided in \cite{Zucker2018}. When $G$ is discrete, $\tS_G(X)$ is just the Stone space of the Boolean algebra of regular open sets in $X$ and the map $\pi_X \colon \tS_G(X) \to X$ sends $p\in \tS_G(X)$ to the unique $x\in X$ such that every regular open neighborhood of $x$ is a member of $p$. Notice that since the regular open algebra of $X$ is a complete Boolean algebra, the space $\tS_G(X)$ is extremally disconnected (and thus never metrizable if $X$ is infinite). We will need the following fact about extremally disconnected spaces.
\vspace{1 mm}

\begin{fact}[Frolík~\cite{Frolik1971}]
  \label{ExtremallyDisc}
  Let $Z$ be a compact extremally disconnected space, and let $f \colon Z\to Z$ be a homeomorphism. Then the set of fixed points of $f$ is clopen.
\end{fact}
\vspace{1 mm}

The following is our main tool for producing free flows.
\vspace{1 mm}

\begin{prop}
  \label{p:EssFreeToFree}
  Let $X$ be a minimal, essentially free $G$-flow. Then $\tS_G(X)$ is free. Moreover, when $G$ is countable, $\tS_G(X)$ admits a metrizable factor which is also free.
\end{prop}

\begin{proof}
  The map $\pi_X \colon \tS_G(X)\to X$, being a morphism of minimal flows, is quasi-open (see the discussion before Proposition~\ref{p:extensions}). It follows that $\tS_G(X)$ must also be essentially free. By Fact~\ref{ExtremallyDisc}, this can only happen if $\tS_G(X)$ is free.

  For the second assertion, proceed as follows. For each non-identity $g\in G$, use the freeness of $\tS_G(X)$ and compactness to find a finite clopen partition $P_g$ of $S_G(X)$ so that for each $A \in P_g$, we have $gA\cap A = \emptyset$. Then use $\{P_g: g\in G\}$ to generate a $G$-invariant subalgebra of the clopen algebra of $\tS_G(X)$. As this subalgebra is countable, the corresponding factor will be metrizable and free.  
\end{proof}
\vspace{1 mm}

Now we can prove the theorems from the introduction that require free flows.
\vspace{1 mm}

\begin{proof}[Proof of Theorem~\ref{th:intro:free-proximal}]
  Let $G$ be icc. By Corollary~\ref{c:ExistsEssentiallyFree}, there exists an essentially free, minimal, proximal flow $G \actson X$. By Proposition~\ref{p:EssFreeToFree}, there exists a metrizable, free, highly proximal extension $X' \to X$. By Proposition~\ref{p:highly-proximal-basic}, $X'$ is minimal and proximal.
\end{proof}
\vspace{1 mm}

\begin{proof}[Proof of Theorem~\ref{th:intro:free-disjoint}]
  \ref{i:thi:free-disjoint-1} By Corollary~\ref{c:disjointness}, there exists an essentially free, minimal, metrizable $Y \perp X$. By Proposition~\ref{p:EssFreeToFree}, there is a metrizable, free, highly proximal extension $Y' \to Y$. By Proposition~\ref{p:highly-proximal-basic}, $Y'$ is minimal and disjoint from $X$.
  
  \ref{i:thi:free-disjoint-2} By Corollary~\ref{c:continuum-many-disjoint}, there exists a family $\set{X_i : i \in 2^{\aleph_0}}$ of essentially free, minimal, metrizable, mutually disjoint $G$-flows. By Proposition~\ref{p:EssFreeToFree}, there exist free, metrizable, highly proximal extensions $X_i' \to X_i$. By Proposition~\ref{p:highly-proximal-basic}, the $X_i'$ are minimal and mutually disjoint.
\end{proof}

We end this section with an application to the universal proximal and strongly proximal flows. The result about $\Pi_s(G)$ was proved in \cite{Breuillard2017} by a different method. 
\begin{prop}
  \label{p:ext-disconnected}
  Let $G$ be a discrete group. Then $\Pi(G)$ and $\Pi_s(G)$ are extremally disconnected.
\end{prop}
\begin{proof}
  It follows from Proposition~\ref{p:highly-proximal-basic} and the universality of $\Pi(G)$ and $\Pi_s(G)$ that  $\tS_G(\Pi(G)) = \Pi(G)$ and $\tS_G(\Pi_s(G)) = \Pi_s(G)$. As universal highly proximal extensions are always extremally disconnected, we have the result.
\end{proof}


\section{The algebra of minimal functions}
\label{sec:min-functions}

Recall that a function $\phi \in \ell^\infty(G)$ is called \df{minimal} if there exists a minimal flow $X$ and a point $x_0 \in X$ such that
\begin{equation}
  \label{eq:gen-fA}
  \phi(g) = f(g \cdot x_0), \quad \text{ for some } f \in C(X).
\end{equation}
Note that in the above definition, instead of considering all possible minimal $G$-flows $X$, we can restrict ourselves to the universal minimal flow $M(G)$. We denote by $\fA(G)$ the closed subalgebra of $\ell^\infty(G)$ generated by all minimal functions.

Next, we explain how to view the algebra $\fA(G)$ as the algebra of functions on a certain compactification of $G$. Denote by $E = E(M(G))$ the \df{enveloping semigroup} of the minimal flow $G \actson M(G)$, i.e., the closure of $G$ in the compact space $M(G)^{M(G)}$ equipped with the product topology. $E$ is a semigroup with the operation of composition of maps and it is \df{right topological}, i.e., for all $q \in E$, the map $E \ni p \mapsto pq \in E$ is continuous. $E$ also acts naturally on $M(G)$ by evaluation. There is a canonical map $G \to E$ given by the action $G \actson M(G)$, which is injective because the action is free (see \cite{Auslander1988}*{Chapter~8}). Thus we can identify $G$ with a subset of $E$. This gives us the following alternative representation of $\fA$.
\vspace{1 mm}

\begin{prop}
  \label{p:env-semigroup-M}
  $\fA(G) = \set{f|_G : f \in C(E)}$.
\end{prop}

\begin{proof}
  For simplicity of notation, we identify $C(E)$ with a subalgebra of $\ell^\infty(G)$ by the inclusion $f \mapsto f|_G$. Suppose first that $\phi \in \ell^\infty(G)$ is as given by \eqref{eq:gen-fA}. Then there exists $x \in M(G)$ and $f \in C(M(G))$ such that $\phi(g) = f(gx)$. By the definition of the topology of $E$, $\phi$ extends to a continuous function on $E$ by the formula $\phi(p) = f(px)$ for $p \in E$. As $C(E)$ is a closed subalgebra of $\ell^\infty(G)$, this gives us that $\fA(G) \sub C(E)$. The reverse inclusion follows from the Stone--Weierstrass theorem: if $p_1 \neq p_2 \in E$, there exists $x \in M(G)$ and $f \in C(M(G))$ such that $f(p_1x) \neq f(p_2x)$ and thus the functions of the form \eqref{eq:gen-fA} separate points in $E$.
\end{proof}
\vspace{1 mm}

We let $\Omega = 2^G$ and identify it with the power set of $G$. We will consider it as a Boolean ring, with addition being the operation $\oplus$ of symmetric difference and multiplication that of intersection. Note that the shift action $G \actson \Omega$ is by ring automorphisms. We continue to call $G$-invariant, closed subsets of $\Omega$ \df{subshifts}. If $X$ and $Y$ are subshifts, we let $X \oplus Y = \set{x \oplus y : x \in X, y \in Y}$ and $XY = \set{xy : x \in X, y \in Y}$. Note that $X \oplus Y$ and $XY$ are also subshifts.

In what follows, it will be more convenient to work with subrings of $\Omega$ rather than subalgebras of $\ell^\infty(G)$. We can identify $\Omega$ with the set of $\{0, 1\}$-valued functions
in $\ell^\infty(G)$ by viewing elements of $\Omega$ as characteristic functions. This gives a natural functorial correspondence between subrings of $\Omega$ and closed subalgebras of $\ell^\infty(G)$ generated by projections, such as $\fA$ (see Lemma~\ref{l:alm-periodic-points-generate}).

Following Furstenberg~\cite{Furstenberg1967}, we will say that a subshift $X \sub \Omega$ is \df{restricted} if for every proper subshift $Y \subsetneq \Omega$, we have that $X \oplus Y \neq \Omega$. Note that if $X \sub Y$ are subshifts and $Y$ is restricted, then so is $X$. It is also clear that if $X$ and $Y$ are restricted, then so is $X \oplus Y$. A point $z \in \Omega$ is \df{restricted} if the subshift $\cl{G \cdot z}$ is restricted. The arguments of \cite{Furstenberg1967} give us the following.
\vspace{1 mm}

\begin{prop}
  \label{p:Furstenberg}
  Let $G$ be an infinite discrete group and let $\Omega$ be as above. Then the following hold:
  \begin{enumerate}
  \item Every minimal subshift is restricted.
  \item If $X$ is a minimal subshift and $Y$ is restricted, then $XY$ is restricted.
  \end{enumerate}
\end{prop}

\begin{proof}
  These two facts are proved in Theorem~III.1 and Proposition~III.2 of \cite{Furstenberg1967}, respectively. While Furstenberg only states these results for $G = \Z$, the only thing about $\Z$ that is used in the proofs is that $\Omega$ is disjoint from all of its minimal subshifts, which holds for all $G$ as we proved in Theorem~\ref{th:main-BDJ}.
\end{proof}
\vspace{1 mm}

Recall that a point $u \in \Omega$ is called \df{minimal} if $\cl{G \cdot u}$ is a minimal subshift. Let $\fB \sub \Omega$ denote the ring generated by  all minimal points. In view of the correspondence alluded to above, we have the following.
\vspace{1 mm}

\begin{lemma}
  \label{l:alm-periodic-points-generate}
  $\fA$ is the closed algebra of $\ell^\infty(G)$ generated by $\fB$ and $\fB = \fA \cap \Omega$. In particular, to prove that $\fA \subsetneq \ell^\infty(G)$, it suffices to show that $\fB \subsetneq \Omega$.
\end{lemma}

\begin{proof}
  Let first $u \in \Omega$ be a minimal point. Let $f_e \in C(\Omega)$ be defined by $f_e(x) = x(e_G)$. Then putting $Z = \cl{G \cdot u}$, we see that $u(g) = f_e(g \cdot u)$ and thus $u \in \fA$. So we conclude that $\fB \sub \fA$. Thus, denoting by $\fA'$ the closed algebra generated by $\fB$, we have $\fA' \sub \fA$.

 To prove the reverse inclusion, recall that $M(G)$, the universal minimal flow of $G$, can be represented as a subset of $\beta G$. In particular, clopen sets separate points in $M(G)$ and thus their characteristic functions generate a dense subalgebra of $C(M(G))$. Let $U \sub M(G)$ be clopen, let $x_0 \in M(G)$ and let $u \in \Omega$ be defined by $u(g) = \chi_U(g \cdot x_0)$ as in \eqref{eq:gen-fA}. It suffices to show that every such $u$ is an element of $\fB$.  
 Define the $G$-map $\pi \colon M(G) \to \Omega$ by $\pi(x)(g) = \chi_U(g \cdot x)$ and note that $u = \pi(x_0)$. Thus the flow $\cl{G \cdot u} = \pi(M(G))$ is minimal and $u \in \fB$.

 It only remains to show that $\fA \cap \Omega \sub \fB$. Observe that $\fA \cap \Omega$ is exactly the set of projections in $\fA$ and it forms a subring of $\Omega$. (Note that for $p, q \in \Omega$, $p \oplus q = p + q - 2pq$; here $+$ and $-$ refer to the operations in the algebra $\ell^\infty(G)$.) It follows from Proposition~\ref{p:env-semigroup-M} that the Gelfand space of the algebra $\fA$ is $E$. The fact that $\fA$ is generated by $\fB$ means that the elements of $\fB$ separate the points of $E$ and $E$ is zero-dimensional; by Stone duality, this implies that $\fB$ corresponds to the ring of clopen subsets of $E$, i.e., the ring of projections in $\fA$.
\end{proof}
\vspace{1 mm}

This leads us to the main theorem of this section.
\vspace{1 mm}

\begin{theorem}
  \label{th:Furst-conjecture}
  Let $G$ be an infinite discrete group. Then $\fA(G) \subsetneq \ell^\infty(G)$.
\end{theorem}
\begin{proof}
  Let $R \sub \Omega$ denote the set of all restricted points. Note that, as no point whose orbit is dense in $\Omega$ belongs to $R$, $R$ is meager. Following \cite{Furstenberg1967}, we see that $\fB \sub R$ (which is enough by Lemma~\ref{l:alm-periodic-points-generate}). Indeed, every point in $\fB$ is of the form $u = \sum_{i=1}^n \prod_{j=1}^{k_i} z_{ij}$ with $z_{ij}$ minimal. Setting $Z_{ij} = \cl{G \cdot z_{ij}}$, we observe that $u \in \sum_{i=1}^n \prod_{j=1}^{k_i} Z_{ij}$ and the latter set is restricted by Proposition~\ref{p:Furstenberg}.
\end{proof}
\vspace{1 mm}

Passing to the duals and combining with Proposition~\ref{p:env-semigroup-M}, the inclusion $\fA(G) \sub \ell^\infty(G)$ yields a surjective map $\beta G \to E(M(G))$. Theorem~\ref{th:Furst-conjecture} then translates to the following.
\vspace{1 mm}

\begin{cor}
  \label{c:env-semigroup-different}
  Let $G$ be an infinite discrete group. Then the natural map $\beta G \to E(M(G))$ is not injective.
\end{cor}


\vspace{2 mm}

We end the section with a question inspired by our techniques. Given a countable discrete group $G$ and $S\subseteq G$, let us say that $S$ is a \df{dense orbit set} if for any minimal $G$-flow $X$ and any $x\in X$, the set $S\cdot x\subseteq X$ is dense. For instance, it follows from Theorem~\ref{th:main-BDJ} that if $Y\subseteq A^{\N\times G}$ is strongly irreducible, $U\subseteq Y$ is non-empty open, and $y\in Y$ has dense orbit, then $\Vis(y, U)\subseteq G$ is a dense orbit set.

\begin{question}
  \label{que:DenseOrbit}
  Characterize the dense orbit sets in countable, discrete groups.
\end{question}

\vspace{2 mm}


\section{The ideal of small sets}\label{sec:small-sets}

In \cite{Glasner1983}, the authors showed that the inclusion $\fA(\Z) \sub \ell^\infty(\Z)$ is proper (again using Furstenberg's work) by proving a more precise result characterizing the \df{interpolation sets} for the algebra $\fA(\Z)$. Given a norm closed, translation invariant subalgebra $\mathcal{A}$ of $\ell^\infty(G)$ containing the constant functions, we say that a subset $A \sub G$ is an 
\df{$\mathcal{A}$-interpolation set} if every bounded real-valued function on 
$A$ can be extended to a function in $\mathcal{A}$. We write $\mathcal{I}_{\mathcal{A}}= \mathcal{I}_{\mathcal{A}}(G)$ for the collection of all $\mathcal{A}$-interpolation sets.

\begin{defn}\label{small}
A subset $B \sub G$ is called \df{small} if the unique minimal subsystem 
in the orbit closure $\overline{G \cdot 1_B} \sub \Omega = \{0,1\}^G$ is
the singleton ${\bf 0}=1_\emptyset$. Equivalently, $B\subseteq G$ is small iff for every finite $F\subseteq G$, the set $\{g\in G: Fg\cap B = \emptyset\}$ is syndetic.  
\end{defn}

The small sets form an \df{ideal} on $G$, that is, the collection of small sets is closed under taking subsets and finite unions. (See, for example, \cite{Bergelson1998} for a proof. In the terminology of \cite{Bergelson1998}, the small sets are precisely the sets which are not \emph{piecewise syndetic}.)

In \cite{Glasner1983}, the authors characterize the interpolation sets for the algebra $\fA(\Z)$ of minimal functions on $\Z$ as precisely the small sets. Below, we generalize this theorem to arbitrary countable groups.

We find it more convenient to use Boolean subalgebras of $\Omega = 2^G$ rather than closed subalgebras of $\ell^\infty(G)$. If $X \sub 2^G$ and $B \sub G$, we will say that $X$ \df{shatters} $B$ if for all $C \sub B$, there exists $x \in X$ such that $x|_B = \chi_C$. It follows from the previous section that $B$ is an $\fA$-interpolation set iff $\fB$ shatters $B$. One way to see this is the following. Recall that the algebra $\fA$ is naturally isomorphic to $C(E)$ where $E = E(M(G))$. Let $B \sub G$ and consider the algebra homomorphisms $C(E) \hookrightarrow \ell^\infty(G) \twoheadrightarrow \ell^\infty(B)$ (the second one being restriction). $B$ is an $\fA$-interpolation set exactly when the composition above is surjective. Dually, this holds iff the composition $\beta B \hookrightarrow \beta G \twoheadrightarrow E$ is injective. Dualizing again (using Stone duality this time), this holds iff the map $\fB \hookrightarrow \Omega \twoheadrightarrow 2^B$ is surjective, i.e., iff $\fB$ shatters $B$.

\begin{theorem}
  \label{th:Interpolation}
  Let $G$ be a countable, infinite, discrete group and let $B \sub G$. Then $B$ is shattered by $\fB$ iff $B$ is small. Moreover, if $B$ is small, there exists a minimal subshift $X \sub \Omega$ which shatters $B$.
\end{theorem}
\begin{proof}
  The proof that if $\fB$ shatters $B$, then $B$ is small is carried out in the same way as in \cite{Glasner1983} and we do not repeat it here. This is the part of the proof that uses the fact that $\Omega$ is disjoint from all minimal flows and Furstenberg's results. The other direction does not directly generalize from \cite{Glasner1983} and we provide a proof using our techniques from Section~\ref{sec:some-prop-clos}. This also gives a new proof for $G = \Z$.
  
  Let $B\subseteq G$ be small and let $\mathcal{S}_B\subseteq K(2^G)$ be the closure of those strongly irreducible subshifts which shatter $B$. This collection is non-empty, as $2^G \in \cS_B$. Moreover, shattering $B$ is a closed condition, so every member of $\mathcal{S}_B$ shatters $B$. We will show that the minimal subshifts 
 in $\mathcal{S}_B$ 
 form a dense $G_\delta$ subset of $\mathcal{S}_B$. We have already seen in Proposition~\ref{p:S-minimal-comeager} that the minimal subshifts form a $G_\delta$ set. So we show that they are dense using a similar argument as in Proposition~\ref{p:S-minimal-comeager}.
	
	Fix $Z'$ a strongly irreducible $B$-shattering subshift, and fix $F\subseteq G$ finite. Find $V\subseteq G$ finite symmetric, containing $e_G$ and $u\in Z'$ such that $u|_V$ contains all $F$-patterns appearing in $Z'$. We also assume that $Z'$ is $V$-irreducible. As in the proof of Proposition~\ref{p:S-minimal-comeager}, we will find a subshift $Y$ with the property that members of $Y$ are at least $V^5$-separated and non-empty, and we will construct the map $\phi\colon Z'\times Y\to 2^G$ and the subshift $Z = \phi(Z'\times Y)$ just as before. However, we need to impose additional conditions on $Y$ in order to ensure that $Z$ shatters $B$.
	
	Since $B$ is small, there is a finite $D\subseteq G$ such that for every $g\in G$, there is $h\in G$ with $V^5h\subseteq Dg$ and $V^5h\cap B = \emptyset$. Let $X\subseteq 2^G$ be the subshift of maximal $D^3$-separated sets, which is strongly irreducible by Lemma~\ref{l:spaced-str-irred}. Now form the shift $D^G$ with alphabet $D$. We define $\psi\colon D^G\times X\to 2^G$ by
    \begin{equation*}
      \psi(w, x)(h) = 1 \iff \exists g \in G \quad x(g) = 1 \And h = w(g)g
    \end{equation*}
    and we set $Y = \mathrm{Im}(\psi)$. In words, the elements of $Y$ are formed by taking a maximal $D^3$-separated set and then moving each member by an element of $D$. By Lemma~\ref{l:strirred closure props}, $Y$ is strongly irreducible, so $Z$ as constructed above is strongly irreducible, and furthermore, every pattern in $S_F(Z) = S_F(Z')$ appears in every member of $Z$. 
	
	We now argue that $Z$ shatters $B$. Let $C\subseteq B$, and using the fact that $Z'$ shatters $B$, find $z\in Z'$ with $z|_B = \chi_C$. Let $x \in X$ be arbitrary. Define $w \in D^G$ as follows: for every $g$ such that $x(g) = 1$, find $h$ such that $V^5h \sub Dg\setminus B$ (this is possible by the choice of $D$) and set $w(g) = hg^{-1}$; if $x(g) = 0$, set $w(g)$ arbitrarily. Then let $y = \psi(w, x)$. By construction,
$y$ has the property that whenever $y(h) = 1$, we have $V^5h\cap B = \emptyset$. It follows that $\phi(z, y)|_B = z|_B = \chi_C$. 
\end{proof}
\vspace{2 mm}


\section{Minimal flows with large automorphism groups}
\label{sec:ManyAuts}

In this section, $G$ is again a countable, discrete group. Given a $G$-flow $Y$, we let $\mathrm{Aut}(Y, G)$ denote the group of $G$-flow automorphisms of $Y$ (that is, all homeomorphisms of $Y$ that commute with the action of $G$). We will generalize the tools from Section~\ref{sec:some-prop-clos} to produce $G$-flows $Y$ such that $\mathrm{Aut}(Y, G)$ embeds any compact metrizable group. Our main tool is a variant of the notion of a strongly irreducible subshift, where we allow the ``alphabet'' to be an arbitrary compact metric space (not necessarily zero-dimensional, as it was in Section~\ref{sec:some-prop-clos}).

Fix $(X, d)$ a compact metric space, and form the $G$-flow $X^G$, where $G$ acts by right shift. Write $\fin{G}$ for the collection of finite subsets of $G$. If $Y\subseteq X^G$ is a subflow and $F\in \fin{G}$, we define $S_F(Y)$, the \emph{$F$-patterns} of $Y$, to be the set $S_F(Y):= \{y|_F: y\in Y\}\subseteq X^F$. Notice that $S_F(Y)$ is compact, and a compatible metric on $S_F(Y)$ is given by 
$$d_F(\alpha_0, \alpha_1) = \max \{d(\alpha_0(f), \alpha_1(f)): f\in F\}.$$
In particular, we can view $S_F(Y)$ as an element of $K(X^F)$, the space of compact subsets of $X^F$ endowed with the Vietoris topology. A metric compatible with this topology is the Hausdorff metric given by 
$$\tilde{d}_F(K, L) := \max \big( \{d_F(y, L): y\in K\} \cup \{d_F(K, z): z\in L\} \big).$$ 
Let $\mathrm{Sub}(X^G)\subseteq K(X^G)$ be the collection of subflows of $X^G$. This is a compact subspace of $K(X^G)$. Another way to view the Vietoris topology on $\mathrm{Sub}(X^G)$ is as follows. The map 
$$\Omega\colon \mathrm{Sub}(X^G)\to \prod_{F\in \fin{G}} K(X^F)$$
given by $\Omega(Y) = (S_F(Y))_{F\in \fin{G}}$ is injective, and the Vietoris topology on $\mathrm{Sub}(X^G)$ makes $\Omega$ an embedding.

Given $\epsilon > 0$ and $F\in\fin{G}$, we say that $Y$ is \emph{$(\epsilon, F)$-minimal} if for every $y\in Y$, $\{(g\cdot y)|_F: g\in G\}$ is $\epsilon$-dense in $S_F(Y)$, meaning that for every $\alpha\in S_F(Y)$, there is $g\in G$ with $d_F(\alpha, (g\cdot y)|_F)< \epsilon$. Notice that $Y$ is $(\epsilon, F)$-minimal iff for some $V\in\fin{G}$ and every $y\in Y$, $\{g\cdot y: g\in G, Fg\subseteq V\}\subseteq S_F(Y)$ is an $\epsilon$-dense set. This is an open condition on $S_V(Y)$, so the $(\epsilon, F)$-minimal flows form an open subset of $\mathrm{Sub}(X^G)$. In particular, the minimal flows form a $G_\delta$ subset, as a flow is minimal iff it is $(\epsilon, F)$-minimal for every $F\in \fin{G}$ and every $\epsilon > 0$.	

We now come to the key definition of this section.
\begin{defn}
	\label{def:PreciselyIrred}
	We say that a subflow $Y\subseteq X^G$ is \emph{precisely irreducible} if there is $V\in \fin{G}$ such that for any $D_0, D_1\in \fin{G}$ which are $V$-apart and any $y_0, y_1 \in Y$, there is $y \in Y$ with $y|_{D_i} = y_i|_{D_i}$ for $i = 0, 1$.
  \end{defn}

  As with strong irreducibility, the requirement that $D_0$ and $D_1$ above be finite is not essential.

If $X$ is finite, the precisely irreducible subshifts are exactly the strongly irreducible ones; if $X$ is a Cantor space, being precisely irreducible is strictly stronger than being strongly irreducible in the sense of Section~5. The disadvantage of this notion compared with strong irreducibility is that the family of precisely irreducible subshifts is not closed under factors; however, this is the correct notion for our purposes here.

Let $\mathcal{S}$ denote the closure of the precisely irreducible subflows ($\cS$ is non-empty because $X^G$ is precisely irreducible). One can show that the minimal flows are dense in $\mathcal{S}$ and below we will prove an ``invariant'' version of this fact.

Let $\Gamma$ be a compact metrizable group and let $d$ be a compatible left-invariant metric on $\Gamma$. We can also endow $\Gamma^G$ with a $\Gamma$-flow structure, where $\Gamma$ acts by left multiplication on each coordinate. This action commutes with the $G$-action, so each $\gamma\in \Gamma$ acts as a $G$-flow automorphism of $\Gamma^G$. If $Y\subseteq \Gamma^G$ is a $G$-subflow which is also $\Gamma$-invariant, then $\Gamma$ also acts on $S_F(Y)$ by left multiplication for each $F\in \fin{G}$. Let $\mathcal{S}_\Gamma$ be the closure of those precisely irreducible subshifts which are also $\Gamma$-invariant; note that $\mathcal{S}_\Gamma\neq \emptyset$, since $\Gamma^G$ is a member. 

The following proposition is an analogue of Proposition~\ref{p:S-minimal-comeager} in this setting.
\begin{prop}
  \label{prop:minimal-dense}
  The minimal flows form a dense $G_\delta$ subset of $\mathcal{S}_\Gamma$.
\end{prop}	

\begin{proof}
  It suffices to show that for every $F \in \fin{G}$ and every $\eps > 0$, the set of $(\eps, F)$-minimal subshifts is open and dense in $\cS_\Gamma$.	As openness was previously discussed, we only show density. The proof mimics the proof of Proposition~\ref{p:S-minimal-comeager} but avoids the use of a ``Conf'' function. Fix $Y$ a precisely irreducible, $\Gamma$-invariant $G$-subflow. By enlarging $F$ and shrinking $\epsilon$ if needed, it is enough to find a precisely irreducible, $\Gamma$-invariant $G$-subflow $Z$ which is $(\epsilon, F)$-minimal and satisfies $\tilde{d}_F(S_F(Y), S_F(Z)) < \epsilon$. 
	
	Find $u\in Y$ and a symmetric $V\in \fin{G}$ with $F\subseteq V$ such that $\{(g\cdot u)|_F: g\in G, Fg\subseteq V\}$ is $\epsilon$-dense in $S_F(Y)$. We may assume that $V$ witnesses the precise irreducibility of $Y$. Define $Z\subseteq \Gamma^G$ by declaring that $z\in Z$ iff there is $y\in Y$ and a maximal $V^5$-separated set $B\subseteq G$ such that all of the following hold:
	\begin{enumerate}
		\item \label{i:md1}
		For every $g\in B$, there exists $\gamma \in \Gamma$ such that $(g\cdot z)|_V = \gamma\cdot u|_V$.
		\item \label{i:md2}
		We have $z|_{G\setminus V^3B} = y|_{G\setminus V^3B}$.
		\item \label{i:md3}
		For $g\in B$, we have $(g\cdot z)|_{V^5}\in S_{V^5}(Y)$.
	\end{enumerate}
	By the precise irreducibility and $\Gamma$-invariance of $Y$, given $y \in Y$ and $B$ as above, there exists $z \in Y$ which satisfies \ref{i:md1} and \ref{i:md2} (and a fortiori \ref{i:md3}). Thus $Z \neq \emptyset$. $Z$ is closed because it is obtained from the closed conditions \ref{i:md1}--\ref{i:md3} by projecting over the compact sets $Y$ and the subshift of maximal $V^5$-separated subsets of $G$. Also, $Z$ is clearly invariant under the actions of both $G$ and $\Gamma$.
    Conditions \ref{i:md2} and \ref{i:md3} ensure that $S_F(Z)\subseteq S_F(Y)$ and the choice of $u$ and \ref{i:md1} ensure that $S_F(Z)$ is $\eps$-dense in $S_F(Y)$. Finally, $Z$ is $(\eps, F)$-minimal by the choice of $u$ and the fact that the sets $B$ are syndetic.
	
	Next we show that $Z$ is precisely irreducible with witness $V^{20}$. Suppose $D_0, D_1 \in \fin{G}$ are $V^{20}$-apart and let $z_i\in Z$, $i = 0, 1$. Find $y_i\in Y$ and maximal $V^5$-separated sets $B_i\subseteq G$ which witness that $z_i\in Z$. Let $C_i = B_i \cap V^5D_i$ and let $D_i' = D_i\cup V^5C_i\subseteq V^{10}D_i$.  Enlarge $C_0\cup C_1$ to a maximal $V^5$-separated set $B\subseteq G$. Using the precise irreducibility of $Y$, find $y\in Y$ with $y|_{D_i'} = y_i|_{D_i'}$ and such that for $g\in B\setminus (D_0'\cup D_1')$, we have $g\cdot y|_V = u|_V$.
	Finally, define $z \in Z$ by $z|_{D_i'} = z_i|_{D_i'}$ and $z|_{G \sminus (D_0' \cup D_1')} = y|_{G \sminus (D_0' \cup D_1')}$. One readily checks that this $z$ verifies the conditions \ref{i:md1}--\ref{i:md3} with $B$ and $y$.    
\end{proof}	
\vspace{1 mm}

Now let $\Gamma = \prod_n U(n)$, where $U(n)$ is the $n$-dimensional unitary group. By the Peter--Weyl theorem, $\Gamma$ embeds every compact metrizable group. Thus we obtain a minimal $G$-flow $Y$ such that $\mathrm{Aut}(Y, G)$ embeds every compact metrizable group. 

We can push this even further by noting that the analogues of Lemma~\ref{l:SDOP-implies-disj-str-irred}, Corollary~\ref{c:disjointness}, and Corollary~\ref{c:continuum-many-disjoint} hold for $\mathcal{S}_\Gamma$. (The only proof which differs from the one in Section~\ref{sec:some-prop-clos} is the one for essential freeness; we prove this in Lemma~\ref{l:PreciseFree} below.) Let $\{Y_i: i< \mathfrak{c}\}\subseteq \mathcal{S}_\Gamma$ be minimal flows such that the product $Y = \prod_i Y_i$ is minimal. Then the group $\Gamma^{\mathfrak{c}}$ embeds in $\mathrm{Aut}(Y, G)$ and as every automorphism of $Y$ lifts to an automorphism of the universal minimal flow $M(G)$ (see \cite{Auslander1988}*{Chapter~10, Exercise~1}), we obtain the following.

\begin{theorem}
	\label{th:ManyAuts}
	Let $G$ be a countable, infinite group. Then $\mathrm{Aut}(M(G), G)$ has cardinality $2^\mathfrak{c}$, the largest possible cardinality. In particular, $M(G)$ is not proximal.
\end{theorem}

That $M(G)$ is not proximal follows from the fact that it has non-trivial automorphisms: if $\gamma \in \Aut(M(G), G)$, then for any point $x \in M(G)$ such that $\gamma \cdot x \neq x$, $x$ and $\gamma \cdot x$ are not proximal.

With a bit more work, we can also generalize a recent result of Cortez and Petite~\cite{Cortez2018}. There, the authors show that every countable, residually finite group can be realized as a subgroup of $\mathrm{Aut}(Y, \Z)$ for some minimal, free $\Z$-flow $Y$ with $Y$ a Cantor set. We generalize this in Theorem~\ref{th:CantorAuts} but first we deal with essential freeness for flows in $\cS_\Gamma$.
\vspace{1 mm}

\begin{lemma}
	\label{l:PreciseFree}
	Let $X$ be a compact metric space with no isolated points. Then the essentially free flows are dense $G_\delta$ in $\mathcal{S}$. If $\Gamma$ is an infinite compact metrizable group, then the same holds for $\mathcal{S}_\Gamma$.
\end{lemma}

\begin{proof}
	Notice that $Y\subseteq X^G$ is essentially free iff for every $g\in G$, every $F\in \fin{G}$, every $\epsilon > 0$, and every $\epsilon$-ball $U\subseteq S_F(Y)$, there is $y\in Y$ with $y|_F\in U$ and $g\cdot y\neq y$. For fixed $g$, $F$, $\epsilon$, and $U$, this is an open condition, so being essentially free is $G_\delta$.
	
	To show that being essentially free is dense, suppose that $Y$ is precisely irreducible with witness $V\in \fin{G}$. Now let
	$$Z = \{z\in X^G: \exists y\in Y \,\,\forall g\in G\,\,\, d(y(g), z(g))\leq \epsilon\}.$$
	Then $Z$ is essentially free (because $X$ has no isolated points) and precisely irreducible, also with witness $V$.
	
	In the case concerning $\mathcal{S}_\Gamma$ (so $X = \Gamma$), the exact same proof works, as the flow $Z$ considered above is $\Gamma$-invariant as long as $Y$ and $d$ are.
\end{proof}
\vspace{1 mm}

\begin{theorem}
	\label{th:CantorAuts}
	Let $G$ be a countable infinite group, and let $H$ be any countable maxap group. Then there is a minimal, free $G$-flow $Y$ with $Y$ a Cantor set such that $H$ embeds into $\mathrm{Aut}(Y,G)$.
\end{theorem}

\begin{proof}
We may suppose that $H\subseteq\Gamma$ with $\Gamma$ a compact metrizable group. Let $X\in \mathcal{S}_\Gamma$ be a minimal, essentially free flow. Then by Proposition~\ref{p:EssFreeToFree}, $\tS_G(X)$ is a minimal, free $G$-flow. Let $B = \mathrm{Clop}(\tS_G(X))$ be the complete Boolean algebra of clopen subsets of $\tS_G(X)$, or equivalently, the regular open subsets of $X$. Notice that $\mathrm{Aut}(X,G)$ acts on $B$ in the obvious way, so $\Gamma$ embeds into $\mathrm{Aut}(\tS_G(X), G)$. Form a subalgebra $B_0\subseteq B$ with the following properties:
\begin{itemize}
	\item 
	$B_0$ is countable.
	\item 
	$B_0$ is $G$-invariant and $H$-invariant.
	\item 
	For each $g\in G$, there is a clopen partition $\tS_G(X) = \bigsqcup_{i<n} A_i$ with each $A_i\in B_0$ so that for each $i< n$, we have $gA_i\cap A_i = \emptyset$.
\end{itemize} 
Then letting $Y$ to be the Stone space of $B_0$, we see that $Y$ is a Cantor set and a minimal, free $G$-flow such that $H$ embeds into $\mathrm{Aut}(Y,G)$.
\end{proof}
\vspace{1 mm}

We end with the following question: is there for every countable, infinite group $G$ a non-trivial minimal incontractible flow? Equivalently, is there a non-trivial minimal flow which is disjoint from every minimal proximal flow?

\bibliography{BDJ}
\end{document}